\documentclass[11pt]{amsart} 

\usepackage{geometry}
\usepackage[titletoc,toc,title]{appendix}
\usepackage{amsfonts}
\usepackage{amsmath}
\usepackage{amsthm}
\usepackage{euscript}
\usepackage{amssymb}
\usepackage{enumerate}
\usepackage{hyperref}
\usepackage{mathrsfs}
\usepackage{xypic}
\usepackage{stmaryrd}
\xyoption{all}

\theoremstyle{plain}
\newtheorem{thm}{Theorem}[section]

\newtheorem{lem}[thm]{Lemma}
\newtheorem{claim}{Claim}
\newtheorem{step}{Step}
\newtheorem{prop}[thm]{Proposition}
\theoremstyle{definition}

\theoremstyle{remark}
\newtheorem{rem}{Remark}[section]
\numberwithin{equation}{section}

\newcommand{\Z}{\mathbb Z}

\newcommand{\C}{\mathbb C}

\newcommand{\N}{\mathbb N}
\newcommand{\p}{\mathbb P}
\newcommand{\G}{\mathbb G}
\newcommand{\PP}{{\mathbb P}}

\newcommand{\Er}{E}

\newcommand{\OO}{\mathcal{O}}
\newcommand{\FF}{\mathcal{F}}
\newcommand{\GG}{\mathcal{G}}
\newcommand{\II}{\mathcal{I}}
\newcommand{\EE}{\mathcal{E}}

\newcommand{\LL}{\mathcal{L}}
\newcommand{\cP}{\mathcal{P}}
\newcommand{\HHH}{\mathcal{H}}
\newcommand{\Q}{\mathcal{Q}}
\newcommand{\cC}{\mathcal{C}}

\newcommand{\pga}{\Big{(}}
\newcommand{\pgc}{\Big{)}}
\newcommand{\ra}{\rightarrow}
\newcommand{\epi}{\twoheadrightarrow}

\newcommand{\transpose}[1]{{#1}^{\intercal}}

\SelectTips {cm}{}

\DeclareMathOperator{\HH}{H} \DeclareMathOperator{\hh}{h}
\DeclareMathOperator{\Ext}{Ext} \DeclareMathOperator{\Hom}{Hom}
\DeclareMathOperator{\im}{Im} 
\DeclareMathOperator{\Ker}{Ker} 
\DeclareMathOperator{\coker}{Coker} 
 
\DeclareMathOperator{\id}{id}
\DeclareMathOperator{\T}{T}

 \DeclareMathOperator{\SL}{\mathrm{SL}}

\DeclareMathOperator{\Kom}{Kom}
\DeclareMathOperator{\K}{K}
\DeclareMathOperator{\D}{D}
\def \p {\partial}
\def \xr {\xrightarrow}
\def \kk {\boldsymbol{k}}
\def \bE {\boldsymbol{E}}
\def \bF {\boldsymbol{F}}
\def \bM {\boldsymbol{M}}

\begin{document}

\title{Linear spaces of matrices of constant rank and instanton bundles}

\author{Ada Boralevi}
\address{Scuola Internazionale Superiore di Studi Avanzati, via Bonomea 265, 34136 Trieste, Italy}
\curraddr{}
\email{ada.boralevi@sissa.it}

\author{Daniele Faenzi}
\address{Universit\'e de Pau et des Pays de l'Adour, Avenue de l'Universit\'e - BP 576 - 64012 PAU Cedex - France}
\curraddr{}
\email{daniele.faenzi@univ-pau.fr}

\author{Emilia Mezzetti}
\address{Universit\`a degli Studi di Trieste, Dipartimento di Matematica e Geoscienze, Via Valerio 12/1, 34127 Trieste, Italy}
\curraddr{}
\email{mezzette@units.it}

\thanks{Research partially supported by the Research Network Program GDRE-GRIFGA. A.Boralevi and E.Mezzetti were supported by MIUR funds, PRIN 2010-2011 
project \lq\lq Geometria delle variet\`a algebriche'', and by Universit\`a degli Studi di Trieste --  FRA 2011. D. Faenzi supported by ANR project GEOLMI} 




\begin{abstract}
  We present a new method to study  $4$-dimensional linear spaces of skew-symmetric matrices of constant co-rank 2, 
  based on rank $2$ vector bundles on $\PP^3$ and derived category tools.
  The method allows one to prove the existence of new examples of size $10 \times 10$ and $14\times 14$ 
  via instanton bundles of charge $2$ and $4$ respectively, and it provides an explanation 
  for what used to be the only known example (Westwick 1996). We also give an algorithm to construct explicitly a matrix of size $14$ of this type. 
\end{abstract}

\maketitle

\section{Introduction}

Given two vector spaces $V$ and $W$ of dimension $n$ and $m$ respectively, we consider a $d$-dimensional vector subspace 
$A$ of $V\otimes W \simeq \Hom (V^*,W)$. 
Fixing bases, we can either write down $A$ as an $m \times n$ matrix whose
entries are linear forms in $d$ variables, or think of $A$ as a matrix whose coefficients are linear forms, i.e. a
map $V^*\otimes\OO_{\PP^{d-1}}\longrightarrow
W\otimes\OO_{\PP^{d-1}}(1)$. 
We say that $A$ has constant rank if every non-zero element of $A$ has
the same rank, or equivalently if the map $V^*\otimes\OO_{\PP^{d-1}}\longrightarrow
W\otimes\OO_{\PP^{d-1}}(1)$, evaluated at every point of $\PP^{d-1}$,
has the same rank.

The interest for this kind of matrices, founded on the classical work by Kronecker and Weierstrass, bears to 
different contexts: linear algebra, theory of degeneracy of vector bundles, study of varieties with
degenerate dual, to cite only a few. 
When analysing these matrices, algebraic geometry appears as a natural
tool, as perhaps first observed by J. Sylvester \cite{sylvester}; for instance characteristic classes of vector
bundles prove to be very useful.
Indeed, since $A$ has constant rank, say $r$,
we obtain a sequence of vector bundles as kernel and cokernel of $A$:
\begin{equation}\label{seq intro}
  0 \to K_A \to V^* \otimes \OO_{\PP^{d-1}}\xr{A} W \otimes
  \OO_{\PP^{d-1}}(1) \to N_A \to 0.
\end{equation}
The computation of Chern classes of these bundles yields restrictions on the values that $r, n, m$ and $d$ can attain. 
Nevertheless, the problem of finding an optimal upper bound on the
dimension of the linear system $A$, given the size and the (constant)
rank of the matrices involved, is widely open in many cases. For a more
extensive introduction to the topic we refer to \cite{Ilic_JM} 
and the numerous references therein, such as \cite{Westwick,Westwick1,Eisenbud_Harris}.
In this work we deal with the special case when $A$ is skew-symmetric.
Very little is known about it, contrary to its symmetric counterpart.  
In our setting the rank $r$ is even and the maximal dimension $l(r,n)$
of a linear subspace $A$ as above  
is comprised between $n-r+1$ and $2(n-r)+1$.
Under the particular hypothesis $r=n-2$, then $3\leq l(r, r+2)\leq 5$, and  
both the kernel and cokernel bundles $K_A$ and $N_A$ have rank 2.
The initial
cases with $n\leq 8$ have been studied in \cite{Manivel_Mezzetti,Fania_Mezzetti}.  
It turns out that $l(4,6)=l(6,8)=3$.

Here, we are mostly interested in the case $d=4$.
When we began our research, the  only known $4$-dimensional space of this kind was one with $r=8$, presented in \cite{Westwick} without any explanation. We reproduce
it in formula \eqref{matrice westwick}.  
The analysis of this ``mysterious'' example has been the starting point
of our investigation. 

For $d=4$, given a skew-symmetric matrix $A$ as above of size $(r+2)$ and constant rank
$r$, one has that $N_A$ and $K_A$ are isomorphic
respectively to $E(\frac r4 +1)$ and $E(-\frac r4)$, where $E$ is a
vector bundle of rank $2$ with $c_1(E)=0$ and $c_2(E)=\frac{r(r+4)}{48}$.
The (twisted) sequence (\ref{seq intro}) then reads:
\begin{equation}
  \label{est}
 \mbox{$0 \to  E(-\frac{r}{4}-2)  \to  \OO_{\PP^3}(-2)^{r+2} \xr{A}
 \OO_{\PP^3}(-1)^{r+2} \to  E(\frac{r}{4}-1)  \to  0.$}
\end{equation}

The main question treated in this paper is how to reverse the
construction, namely, how to start from a vector bundle $E$ of rank $2$
on $\PP^3$ and obtain a skew-symmetric matrix $A$ of linear forms on $\PP^3$
having constant co-rank $2$. ($A$ will then have $E$ as kernel, up to a twist by a line bundle.)
In fact, besides $E$, one more ingredient is needed, namely a class
$\beta \in \Ext^2(E(\frac r4 -1),E(-\frac r4 -2))$ corresponding to
an extension of type \eqref{est}.
This is where the first new tool from algebraic geometry comes
into play: derived category theory.
The main idea is that, given $E$ with the allowed Chern classes
$c_1(E)=0$ and $c_2(E) = \frac{r(r+4)}{48}$, and a class $\beta$ as above, 
we can obtain a $2$-term complex $\cC$ as cone of $\beta$, interpreted as
a morphism $E(\frac r4 -1) \to E(-\frac r4 -2)[2]$
in the derived category $\D^b(\PP^3)$.
The next step entails using Beilinson's theorem to show that,
under some non-degeneracy conditions of the maps
$\mu^p_t:\HH^p(E(\frac r4 -1+t)) \to \HH^{p+2}(E(-\frac r4 -2+t))$ induced by $\beta$,
the complex $\cC$ is of the form $\OO_{\PP^3}(-2)^{r+2} \to
\OO_{\PP^3}(-1)^{r+2}$, sitting in degrees $-2$ and $-1$. The desired 
matrix $A$ then appears as differential of $\cC$.
To complete the argument, we show that $A$ is necessarily skew-symmetrizable. 
This is the content of the main Theorem \ref{nec&suff}. For the reader not familiar with the subject, 
Section \ref{idea kuz gen} also contains a brief review of the most important features of derived category theory, 
with emphasis on the tools and techniques that are used in this paper. 

Our next results aim at reducing the number of requirements on the maps $\mu^p_t$'s by imposing convenient assumptions on $E$.
First, in Theorem \ref{conto kuz gen} we show that these conditions can be significantly
simplified if $E$ has natural cohomology, i.e. $\HH^p(E(t))\ne 0$ for at most one $p$, for all $t$.
Here is where the next algebro-geometric ingredient comes in, 
namely  instanton bundles (``$k$-instantons'' if $c_2(E)=k$), 
first introduced in \cite{adhm}.
Indeed, general instantons have natural cohomology, see \cite{HH}.
Next, we use the description of the minimal graded free
resolution of a general instanton $E$ and of its cohomology module
$\bigoplus_{t \in \Z} \HH^2(E(t))$ (cf. \cite{raha, Decker}) to
further reduce the requirements on the $\mu^p_t$'s to a {\it single condition}. 
This is done in Theorem \ref{starstar gratis}.

Let us outline the main applications we draw from this method.
\begin{enumerate}
\item Westwick's example is given an explanation in terms of
  $2$-instantons.
  Taking advantage of the extensive literature  
(for example \cite{Hart_ist,Newstead,costa_ottaviani}) 
on the moduli space $M_{\PP^3}(2;0,2)$ of rank $2$ stable vector
bundles on $\PP^3$ with Chern classes $c_1=0$ and $c_2=2$, we determine that the
``Westwick instanton'' belongs to the most special orbit of
$M_{\PP^3}(2;0,2)$ under the natural action of $\SL(4)$, see
Theorem \ref{orbita speciale}.
\item We show the existence of a continuous family of examples of $10 \times 10$ matrices of
  rank $8$, all non-equivalent to Westwick's one.
  These are obtained by showing that all 2-instantons
  have classes $\beta$ satisfying the required condition, see Theorem \ref{2-ist}.
\item We show the existence of a continuous family of new examples of $14 \times 14$ matrices of rank $12$, starting with general $4$-instantons, see Theorem
  \ref{4-ist}.
\item We exhibit an explicit example of a $14 \times 14$ skew-symmetric matrix of
  constant rank $12$ in $4$ variables, together with an
  algorithm capable of constructing infinitely many of them, see the Appendix.
\end{enumerate}

In all instances above, $\beta$ is a general element of
$\Ext^2(E(\frac r4 -1),E(-\frac r4 -2))$.
This is not at all surprising: the non-degeneracy conditions on the $\mu^p_t$'s are
open, so once an element $\beta_0$ satisfies the requirements, the same will hold for the general element $\beta$.
Still, the construction of examples of higher size does not seem straightforward. However, we
believe that this method is promising, and that it can be potentially applied to other varieties
than $\PP^3$, and to other linear subspaces than the one we
considered. We would also like to develop algorithmic methods to
find  explicit expressions of our new examples, generalising what is
done in the Appendix. We are currently working on these further applications.

\vspace{0,2cm}

The paper is organised as follows: in Section \ref{general} we introduce the problem and the state of the art, 
focusing on the case of $4$-dimensional linear subspaces of skew-symmetric matrices of constant co-rank $2$. 
Section \ref{idea kuz gen} contains our main results.
Here we characterise these spaces in terms of necessary and sufficient
conditions to impose on the maps $\mu^p_t$ introduced above.
In Section \ref{sezione istantoni generali} we show how and why the most natural candidates for our purposes are general instantons. 
The last Sections \ref{caso 8} and \ref{caso 12} are devoted to show the existence 
of new examples. Finally in the Appendix we give an algorithmic method that explicitly gives a new matrix of size $14$.

\vspace{0,2cm}

\noindent \textbf{Acknowledgements.} The authors are indebted to A. Kuznetsov
for suggesting the use of derived categories, and for directing us
towards the main idea.
In particular the first named author is grateful for the interesting discussions had in Trieste.\\ 
The authors would also like to thank the referees for their valuable suggestions.

\section{General set-up and the co-rank two case}\label{general}

Here we introduce some preliminary material. We work over the field
$\C$ of complex numbers, although all the constructions seem to carry
over smoothly to any algebraically closed field of characteristic other than $2$.
Given a vector space $W$ over $\C$, of dimension $d$, we denote by $W^*= \Hom(W,\C)$ its dual, and 
we fix a determinant form so that $W \simeq W^*$.
The projective space $\PP^{d-1}=\PP W$ is the space of lines through $0$, thus $\HH^0(\OO_{\PP W}(1)) = W^*$. 

Any claim about a {\it general element} in a given parameter space
means that the claim holds for all elements away from a countable
union of Zariski closed subsets of the parameters.

\subsection{General set-up}\label{general set-up}

Let $V$ be a vector space of dimension $n$ over $\C$, and let $A \subseteq V \otimes V$ be a linear subspace of dimension $d$. 
We say that $A$ has \emph{constant rank} $r$ if every non-zero element of $A$ has rank $r$. Given a basis for $V$, 
we can write down $A$ as a $n \times n$ matrix of linear forms in $d$ variables, that we denote by the same letter $A$.

The matrix $A$ can be viewed naturally as a map $V^* \otimes \OO_{\PP A}\xrightarrow{A} V \otimes \OO_{\PP A}(1)$, inducing the following exact sequence, where 
$K_A$ and $N_A$ denote the kernel and cokernel of $A$, and $\EE_A$ its image:
\begin{equation}\label{se gen}
  \xymatrix@C-2ex@R-3ex{0 \ar[r] & K_A \ar[r] & V^* \otimes \OO_{\PP A}\ar[rr]^-A \ar[dr] && V \otimes \OO_{\PP A}(1)\ar[r]&N_A \ar[r]&0\\
    &&&\EE_A \ar[dr] \ar[ur]&&&\\
    &&0 \ar[ur]&&0&&\\}.
\end{equation}
From the assumption that $A$ has constant rank $r$ we deduce that $K_A,N_A$ and $\EE_A$ are vector bundles on $\PP A$. 
(That are of rank $n-r$, $n-r$ and $r$ respectively.)

A computation of invariants of these vector bundles shows that there is a bound on the maximal dimension $l(r,n)$ that 
such a subspace $A \subseteq V \otimes V$ can attain. Westwick \cite{Westwick1} proved that for $2 \leq r \leq n$: 
\begin{equation}\label{bounds}
  n-r+1 \leq l(r,n) \leq 2(n-r) +1.
\end{equation}
Moreover for a given $d=\dim A$ with $n-r+1 \le d \leq 2(n-r) +1$, only
some values of $r$ are allowed. We refer to \cite{Ilic_JM,Westwick1,Westwick} for further details.

We now focus on the special case $r=n-2$.
In this case $K_A$ and $N_A$ are vector bundles of rank $2$, and the
bounds (\ref{bounds}) on the maximal dimension of $A$ become $3 \leq
l(r,r+2) \leq 5$. 

Specialising even more, let us now suppose that the subspace $A$ lies either in $\wedge^2 V$ or in $S^2 V$. Then the 
skew-symmetry (or the symmetry) of the matrix yields a symmetry of the
exact sequence (\ref{se gen}). In particular: $N_A \simeq K^*_A(1)$,
and  
$\EE_A^* \simeq \EE_A(1)$.
The same computation of invariants as above entails that the first
Chern class $c_1(\EE)= c_1(K^*)= \frac{r}{2}$, and thus the rank $r$
is even. Moreover from the rank $2$ hypothesis we get that $K_A^*\simeq
K_A(\frac{r}{2})$.

\begin{rem}\label{fibrati iso}
  The natural action of $\SL(n)$ on $V$ extends to an action on the linear subspaces of $\wedge^2V$ (or $S^2 V$). Let $g\in \SL(n)$ act on a subspace $A$, and denote $A'=g A$. 
  Then we have a commutative diagram:
  \[
  \xymatrix
  {0\ar[r] &K_A \ar[r]&V^*\otimes\OO_{\PP A} \ar[r]^-{A}& V\otimes\OO_{\PP A}(1)\ar[r] &K_A^*(1)\ar[r]& 0\\
    0 \ar[r] &K_{A'} \ar[r]\ar[u]^-{g^*}&V^*\otimes\OO_{\PP A'} \ar[u]^-{g^*}\ar[r]^-{A'} &V\otimes\OO_{\PP A'}(1)\ar[u]^-{g^*} \ar[r]&K^*_{A'}(1)\ar[r]\ar[u]^-{g^*}&0}
  \]
  which shows that the kernel bundles $K_A$ and $K_{A'}$ are isomorphic.
  Similarly, if $h\in \SL(d)$, then $h: A\rightarrow A$ induces:
  \[
  \xymatrix
  {0\ar[r] &K_A \ar[r]&V^*\otimes\OO_{\PP A} \ar[r]^-{A}& V\otimes\OO_{\PP A}(1) \ar[r]&K_A^*(1)\ar[r]& 0\\
    0 \ar[r] &K_{(h A)} \ar[r]\ar[u]^-{h^*}&V^*\otimes\OO_{\PP(h A)} \ar[r]^-{h A} \ar[u]^-{h^*}&V\otimes\OO_{\PP(h A)}(1)\ar[u]^-{h^*} \ar[r]&K^*_{(h A)}(1)\ar[u]^-{h^*}\ar[r]&0}
  \]
  so $K_A=h^*(K_{h A}).$

  Altogether, there is an action of $\SL(n) \times \SL(d)$ on the
  linear spaces of skew-symmetric (or symmetric)  matrices of
  constant rank.
  If two matrices $A$ and $B$ are equivalent under
  this action,  
  the corresponding vector bundles $K_A$ and $K_B$ will 
  belong to the same orbit under the action of $\SL(d)$.
\end{rem}

\subsection{Skew-symmetric matrices of constant co-rank two}\label{skew-symm set-up}
We are interested in the case where $A \subseteq \wedge^2 V$, i.e. linear spaces of skew-symmetric matrices of constant
co-rank $2$. The case $6 \times 6$ is treated in
\cite{Manivel_Mezzetti}, and the case $8 \times 8$ in
\cite{Fania_Mezzetti}. In both instances the maximal dimension of $A$
is $3$. 

The state of the art when we began working on the subject---at the
best of our knowledge---was the following: many examples were known
for $d=3$,  
no examples for $d=5$, and only one example for $d=4$, appearing in
\cite{Westwick}. The example has $r=8$, which is  
the smallest value allowed for dimension $d=4$.

From now on we work on the case $d=4$. $A$ will always denote a skew-symmetric matrix of linear forms in $4$ variables,
having size $r+2$ and constant rank $r$, kernel $K_A$ and cokernel $K_A^*(1)$.

\begin{lem}
Let $A$ be as above. Then $K_A \simeq E(-\frac r4)$, where $E$ is an indecomposable rank $2$ vector
bundle on $\PP^3$ with Chern classes:
\[
c_1(E)=0 \quad \hbox{and} \quad c_2(E)=\frac{r(r+4)}{48}.
\]
It follows that $r$ is of the form $12s$ or $12s-4$, for some $s \in \N$.
\end{lem}

\begin{proof}
  We have already remarked that the condition $c_1(K^*_A)= \frac{r}{2}$ is entailed by the invariants of the bundles in (\ref{se gen}). 
  Imposing the further condition that the cokernel is isomorphic to $K^*_A(1)$ forces the Chern polynomials to satisfy the equality 
  $c_t(K_A^*(1))=(1+t)^{r+2}c_t(K_A)$. From this we get $c_2(K_A)=\frac{r(r+1)}{12}$, and a direct computation then 
  shows that $E=K_A(\frac r4)$ has the desired Chern classes.
  Moreover since the Chern polynomial of $E$ is irreducible over $\Z$, we deduce that $E$ is indecomposable.
\end{proof}

The exact sequence (\ref{se gen}) can then be written as follows: 
  \begin{equation}\label{se E} 
    \mbox{$0 \to E(-\frac{r}{4}) \to  \OO_{\PP^3}^{r+2}
      \xr{A}    \OO_{\PP^3}(1)^{r+2}\to E(\frac{r}{4}+1)
      \to 0$}.
  \end{equation}
It will be useful to have Riemann-Roch formula at hand, see for
instance \cite[Appendix A]{Har}.
It implies that the algebraic Euler characteristic
$\chi(E)=\sum_i(-1)^i\hh^i(E)$ of a vector bundle $E$ of rank $2$ on
$\PP^3$ with Chern classes $c_1$, $c_2$ is:
\[
\chi(E)=\frac 16 {c}_{1}^{3}- \frac 12 {c}_{1} {c}_{2}+{c}_{1}^{2}-2
{c}_{2}+ \frac {11}6
      {c}_{1}+2.
\]

\section{Main construction}\label{idea kuz gen}

In this section we will state and prove our main result, Theorem
\ref{nec&suff}.
This result establishes a necessary and sufficient condition for a
rank-$2$ vector bundle on $\PP^3$, together with a certain cohomology
class, to give a skew-symmetric matrix of constant co-rank $2$.
The results of this section continue to hold if we replace $\C$ with
an algebraically closed field of characteristic different from $2$.

\subsection{Necessary conditions}

We work in the setting described in Section \ref{skew-symm set-up}:
assume that $A \subseteq \wedge^2 V$ is a $4$-dimensional linear
subspace  
of skew-symmetric matrices of size $r+2$ and constant rank $r$. 
Take the sequence (\ref{se E}) and tensor it by $\OO_{\PP^3}(-2)$:
\begin{equation} \label{se E twist}
  \xymatrix@C-2.5ex@R-4ex{0 \ar[r] & E(-\frac r4-2) \ar[r] & \OO_{\PP^3}(-2)^{r+2}\ar[rr]^-A \ar[dr] && \OO_{\PP^3}^{r+2}(-1)\ar[r]&E(\frac r4-1) \ar[r]&0\\
    &&&\EE(-2) \ar[dr] \ar[ur]&&&\\
    && 0 \ar[ur]&&0&&\\}
\end{equation}

The $4$-term sequence above corresponds to
an element:
\[
\mbox{$\beta \in \Ext^2(E(\frac{r}{4}-1),E(-\frac{r}{4}-2)).$}
\]
For all integers $t,p$, the composition of boundary
maps in the cohomology long exact
sequence of \eqref{se E twist} gives maps:
\[
\begin{array}{l}
\mu^p_t : \HH^p(E(\frac{r}{4}-1+t)) \to \HH^{p+2}(E(-\frac{r}{4}-2+t)).
\end{array}
\]
The maps $\mu^p_t$ can also be thought of as the cup product
with the cohomology class $\beta$ in the Yoneda product:
\[
\mbox{$
\Ext^2(E(\frac{r}{4}-1),E(-\frac{r}{4}-2)) \otimes
\Ext^p(\OO_{\PP^3}(-t),E(\frac{r}{4}-1)) \to
\Ext^{p+2}(\OO_{\PP^3}(-t),E(-\frac{r}{4}-2))$}.
\]

\begin{lem}
  The exact sequence \eqref{se E twist} gives the following: 
  $$  (\star)\:\:\:\:\left\{ \begin{array}{ll}
      \mu^p_0:\HH^p(E(\frac{r}{4}-1))\simeq \HH^{p+2}(E(-\frac{r}{4}-2)), &p=0,1;\\
      \HH^q(E(\frac{r}{4}-1))= \HH^{q-2}(E(-\frac{r}{4}-2))=0, &q=2,3.
    \end{array}\right.$$
\end{lem}

\begin{proof}
This is a direct consequence of the fact that the cohomology of both $\OO_{\PP^3}(-2)$ and $\OO_{\PP^3}(-1)$ 
vanishes in all degrees. 
The vanishing $\HH^0(E(-\frac{r}{4}-2))=\HH^3(E(\frac{r}{4}-1))=0$ are immediate. The same holds for 
$\HH^0(\EE(-2))=\HH^3(\EE(-2))=0$, which in turn implies
$\HH^1(E(-\frac{r}{4}-2))=\HH^2(E(\frac{r}{4}-1))=0$.
From the rightmost 
short exact sequence in \eqref{se E twist}
we deduce the isomorphisms $\HH^p(E(\frac{r}{4}-1))\simeq \HH^{p+1}(\EE(-2))$, for $p=0,1$, while from 
the leftmost one we get, for the same values of $p$, the isomorphism $\HH^{p+2}(E(-\frac{r}{4}-2)) \simeq \HH^{p+1}(\EE(-2))$.
\end{proof}

\begin{lem}   The exact sequence \eqref{se E twist} gives the following: 
$$  (\star\star)\:\:\:\:\left\{ \begin{array}{ll}
    \mu^0_1:\HH^0( E (\frac{r}{4} )) \epi \HH^2 ( E ( -\frac{r}{4} -1));&\\
    \mu^1_1:\HH^1(E(\frac{r}{4})) \simeq \HH^3 ( E (-\frac{r}{4}-1));&\\
    \HH^2(E(\frac{r}{4}))=\HH^3(E(\frac{r}{4}))=\HH^0(E(-\frac{r}{4}-1))=0.
  \end{array}\right.$$
\end{lem}

\begin{proof}
We apply the argument of the previous Lemma to (\ref{se E}) twisted by
$\OO_{\PP^3}(-1)$ instead of $\OO_{\PP^3}(-2)$. 
The vanishing $\HH^2(E(\frac{r}{4}))=\HH^3(E(\frac{r}{4}))=\HH^0(E(-\frac{r}{4}-1))=0$ is immediate, as in the previous proof.
Moreover the cohomology of the sequence on the left-hand side of
\eqref{se E twist} yields $\HH^i(E(-\frac{r}{4}-1)) \simeq \HH^{i-1}(\EE(-1))$ for $i=1,2$ and $3$, and 
the sequence on the right-hand side this time gives an isomorphism $\HH^1(E(\frac{r}{4}))\simeq \HH^2(\EE(-1))$, and also
an exact sequence:
\[
\mbox{$0 \to \HH^0(\EE(-1)) \to \C^{r+2} \to \HH^0(E(\frac{r}{4})) \to \HH^1(\EE(-1)) \to 0$.}
\]
This concludes the proof.
\end{proof}

\subsection{The conditions are sufficient}

Our first main result is that conditions $(\star)$ and $(\star\star)$ are not only necessary, but also sufficient.

\begin{thm}\label{nec&suff}
  Let $r$ be a fixed integer number of the form $12s$ or $12s-4$, $s \in \N$. Let $E$ be a rank 2 vector bundle on $\PP^3$, 
  with $c_1(E)=0$ and $c_2(E)=\frac{r(r+4)}{48}$. There exists a
  skew-symmetric matrix $A$ of linear forms, having size $r+2$, constant rank $r$, 
  and $E(-\frac{r}{4}-2)$ as its kernel, if and only if there exists $\beta \in \Ext^2(E(\frac{r}{4}-1),E(-\frac{r}{4}-2))$ 
  that induces $(\star)$ and $(\star\star)$.
\end{thm}

The main idea in the proof of this result is the following.
First, look at the element
$\beta \in \Ext^2(E(\frac{r}{4}-1),E(-\frac{r}{4}-2))$ as a morphism in
$\D^b(\PP^3)$, the derived category of $\PP^3$, so $\beta :
E(\frac{r}{4}-1) \to E(-\frac{r}{4}-2)[2]$.
Second, use
Beilinson Theorem to show that the cone of the morphism $\beta$ is a
$2$-term complex of the form
$\p:\OO_{\PP^3}(-2)^{r+2} \to \OO_{\PP^3}(-1)^{r+2}$, having
$E(-\frac{r}{4}-2)$ as kernel and 
$E(\frac{r}{4}-1)$ as cokernel.
Third, prove that the differential $\p$ is
skew-symmetric, so $\p$ is nothing but the matrix $A$ we are looking
for. 

The advantage of working with derived categories consists in allowing us 
to deal more comfortably with complexes, which is precisely what we need, since $A$ appears as differential of a $2$-term complex.
As a by-product, derived categories allow cleaner and more conceptual statements,
avoiding cumbersome tracking of indexes in spectral sequences.

\subsubsection{A quick tour of derived categories I. General features}

For the reader's convenience, we give a quick account of the language of derived categories,
trying to avoid excessive technicalities.
For more precise definitions and results, we refer to the excellent book \cite{libro_Huybrechts}, in particular Chapters 1-3 and 8.
We use some additional
remarks from \cite{Gelfand_Manin} and \cite{maclane}.
For a short but enlightening introduction to derived categories, see also \cite{caldararu:skimming}.

\vskip.1in

We give some elements of the definition of the derived
category $\D^b(X)$ of bounded complexes of coherent sheaves 
over a smooth projective variety $X$, defined over $\C$.
As an auxiliary tool, we also describe quickly the homotopical category $\K(X)$.

The first ingredient of a category are objects. For this, denote by
$\Kom(X)$ the abelian, $\C$-linear category of complexes of coherent
sheaves on $X$.
The objects in $\D(X)$ and $\K(X)$ are the same as those in $\Kom(X)$, i.e., complexes $\FF$ of the following form:
\[
\mbox{$\FF:\:\:\:\:\:\: \ldots \ra F_{i-1} \xrightarrow{\p_{i-1}} F_i \xrightarrow{\p_i} F_{i+1} \ra \ldots$}
\]
where the $F_i$'s are coherent sheaves on $X$, and $\p_{i+1} \circ
\p_i=0$ for all $i \in \Z$.
We say that $\FF$ lies in $\D^b(X)$ if $\FF$ is bounded in both directions.
We have a {\it shift functor} $\FF \mapsto \FF[1]$, given by shifting
degrees of one place to the left, so
$F[1]_i:=F_{i+1}$.
For the differential of $\FF[1]$, as well as for other sign
conventions, we follow the standard agreement (or at least the same as \cite{libro_Huybrechts}), so $\p_i^{\FF[1]}=-\p^\FF_{i+1}$.

Let us now turn to morphisms in $\K(X)$.
One starts with morphisms in $\Kom(X)$, i.e. \emph{chain maps} $f: \FF
\ra \GG$, that is, collections of maps $f_i:F_i \ra G_i$ such that the
obvious squares commute.
Then, one considers \emph{homotopically equivalent} morphisms
$f,g:\FF \ra \GG$, i.e. such that there exists a collection of morphisms $h_i:F_i \ra G_{i-1}$, for all $i$, such that $f_i-g_i=h_{i+1} \circ \p^\FF_i + \p^\GG_{i-1} \circ h_i$. 
Denote by $\K(X)$ the \emph{homotopy category} of $X$: morphisms in $\K(X)$ are chain maps of complexes, modulo homotopy equivalence. 

An important feature of $\K(X)$ is that it has the structure of a triangulated category.
This means that there is a collection, modeled on exact sequences, of {\it distinguished
  triangles}, i.e. triples of complexes and morphisms as in the diagram:
\begin{equation}
  \label{triangolozzo}
\FF  \xrightarrow{f} \GG \xrightarrow{g} \EE \xrightarrow{h} \FF[1],
\end{equation}
satisfying certain axioms (see \cite[Chapter 1]{libro_Huybrechts}).
A triangle is distinguished in $\K(X)$ if and
only if it is
isomorphic in $\K(X)$ to the cone triangle associated with $f:\FF \ra
\GG$, where the complex $\EE=\cC(f)$ is defined by: 
\begin{center}
  $E_i:=F_{i+1} \oplus G_i$, with differential
$\p_i^{\EE}:=\begin{bmatrix}
  -\p^\FF_{i+1} & 0\\
  f_{i+1} & \p^\GG_i
  \end{bmatrix}$,
\end{center}
and with the obvious maps $g: \GG \ra \EE$ and $h: \EE \ra \FF[1]$. 
Given a complex $\FF$, its \emph{cohomology
sheaves}
$\HHH^i(\FF)$ are defined as:
\[\mbox{$\HHH^i(\FF):=\Ker (\p^i) /\im (\p^{i-1}$}).\] 
Any distinguished triangle \eqref{triangolozzo} gives a long
cohomology sequence:
\[
\cdots \ra \HHH^i(\FF) \ra \HHH^i(\EE) \ra \HHH^i(\GG) \ra
\HHH^{i+1}(\FF) \ra \cdots.
\]
A chain map $f: \FF \ra \GG$ induces maps $\HHH^i(f):\HHH^i(\FF) \ra \HHH^i(\GG)$,
and $f$ is called a \emph{quasi-isomorphism} if, for all $i$, the map
$\HHH^i(f)$ is an isomorphism. 

We are now ready to introduce morphisms $\FF \to \GG$ in the derived category
$\D(X)$.
These are equivalence classes of diagrams of the form:
\begin{equation}\label{roof}
  \xymatrix@-2ex{   &\EE\ar[dl]_f\ar[dr]^g&\\
            \FF \ar@{..>}[rr]&                 &\GG}
\end{equation}
where both arrows $f$ and $g$ represent morphisms in $\K(X)$ and
$f:\EE \ra \FF$ is a quasi-isomorphism. Such diagrams are called \emph{roofs}.
In other words, we formally invert quasi-isomorphisms, so a roof $\FF
\to \GG$ as above can be thought of as $g/f$, with a convenient formalism.
This process is called \emph{localisation}, by analogy with the process of localisation of
rings along a multiplicative system.
The category $\D(X)$ inherits from $\K(X)$ the structure of a
triangulated category; in particular the notions of shift, cone of a morphism, and distinguished triangle are well-defined.
Note that the cone of a morphism in $\D(X)$ is defined up to an
isomorphism which is not unique in general.

Coherent sheaves on $X$ are elements of $\D(X)$, concentrated in a
single degree. We will usually take this degree to be zero, following the standard convention. 
Morphisms of coherent sheaves can be seen as complexes
whose cohomology is concentrated in two consecutive degrees.
Indeed, given a complex $\FF$ having cohomology in degrees $-2$ and
$-1$ only, we can replace $\FF$ with:
\[
\FF': \qquad \cdots \to F_{-3} \to F_{-2} \to \Ker (\partial_{-1})
\]
where the map $F_{-2} \to \Ker (\partial_{-1})$ is induced by the composition $F_{-2} \to \im(\partial_{-2}) \to \Ker( \partial_{-1})$. 
The induced chain map $\FF' \to \FF$ is a quasi-isomorphism. Then, one replaces $\FF'$ with:
\[
\FF'': \qquad \coker(\partial_{-3}) \to \Ker(\partial_{-1}),
\]
where the differential is the composition of the surjection $\coker(\partial_{-3})
\to \im(\partial_{-2})$ with the injection
$\im(\partial_{-2}) \to \Ker(\partial_{-1})$.
This time, we
obtain a chain map $\FF'' \to \FF'$, which is again a quasi-isomorphism.
Altogether, we get a roof $\FF \leftarrow \FF' \to \FF''$, so $\FF$ is quasi-isomorphic to a complex
with two terms only.

Having introduced derived categories, let us now briefly introduce derived functors.
In the derived category, a coherent sheaf $\FF$ is the same thing as any
resolution of $\FF$. (We think essentially of injective resolutions,
with a slight abuse of terminology since we have to rely on
quasi-coherent sheaves as well to perform this.) Then, taking global sections of (an injective resolution of) $\FF$ results in a complex whose
$i$-th cohomology is $\HH^i(\FF)$.
Likewise, an object $\FF$ in $\D(X)$ is equivalent to the total
complex attached to a resolution of each of the $F_i$. By taking global sections of this total complex, we get a complex of
vector spaces: the $i$-th cohomology of this complex is called the
\emph{hypercohomology} $\HH^i(\FF)$ of $\FF$. 
Given an exact triangle \eqref{triangolozzo}, we have the
hypercohomology long exact sequence (see \cite[Section 2.2]{libro_Huybrechts}):
\[
\cdots \ra \HH^i(\FF) \ra \HH^i(\EE) \ra \HH^i(\GG) \ra
\HH^{i+1}(\FF) \ra \cdots
\]
Moreover, given complexes $\FF$ and $\GG$ in $\D^b(X)$, we will have
to consider the groups $\Ext^i(\FF,\GG)$ in the category $\Kom(X)$,
namely the $i$-th cohomology of the total complex obtained by applying
$\Hom(-,\GG)$ to (an injective resolution of) $\FF$.
It turns out that this amounts to compute
morphisms of shifted complexes in the derived category (see \cite[Rmk 2.57]{libro_Huybrechts}):
\begin{equation}\label{ext=hom}
  \Ext^i(\FF,\GG) \simeq \Hom_{\D^b(X)}(\FF,\GG[i]).
\end{equation}
Also in this case, we get a long cohomology exact sequence by applying
$\Hom_{\D^b(X)}(-,\GG)$ to a distinguished triangle.
The same considerations apply to other classical functors, such
as tensor product, local cohomology, higher direct images, and so forth.

\subsubsection{A quick tour of derived categories II. Beilinson theorem}

We now focus our attention on $\D^b(\PP^n)$, the bounded derived
category of the projective space. Its main feature is Beilinson
Theorem, which states that 
$\D^b(\PP^n)$ is generated by the exceptional collection $\langle
\OO_{\PP^n}(-n),\OO_{\PP^n}(-n+1),\ldots,\OO_{\PP^n}(-1),\OO_{\PP^n}
\rangle$, with dual collection $\langle
\OO_{\PP^n}(-1),\Omega^{n-1}_{\PP^n}(n-1), \ldots,
\Omega^1_{\PP^n}(1), \OO_{\PP^n} \rangle$, cf.
\cite[Coroll. 8.29]{libro_Huybrechts}.
We need the following version:

\begin{prop}[Beilinson Theorem]\label{beilinson}
Let $\FF$ be a bounded complex of coherent sheaves on $\PP^n$.
Then there exists a complex $\LL$, whose factors are  
$L_k:=\bigoplus_{s-j=k}\HH^s(\FF \otimes \Omega_{\PP^n}^{j}(j)) \otimes \OO_{\PP^n}(-j)$, which is quasi-isomorphic to $\FF$.
\end{prop}

We call $\LL$ the {\it decomposition of $\FF$}, and, for fixed $j$, we call the terms
$\HH^s(\FF \otimes \Omega_{\PP^n}^{j}(j)) \otimes \OO_{\PP^n}(-j)$ the {\it components} of
$\FF$ along $\OO_{\PP^n}(-j)$.
The theorem of Beilinson was proved in \cite{beilinson}, see also
 \cite[Prop. 8.28]{libro_Huybrechts}, \cite{caldararu:skimming}.
Although our statement is slightly more general than in
\cite{libro_Huybrechts}, since we take into account complexes and not
just sheaves, the proof
goes through verbatim.
It will be useful to have a graphic description of Beilinson
Theorem. Consider the $(n+1) \times (n+1)$ square diagram, sometimes referred
to as the {\it Beilinson table} of $\FF$:
\begin{equation}\label{quadrato2}
\begin{array}{|c||c|c|c|c|c|c|}
  \hline
    \hh^n(\FF \otimes \Omega_{\PP^n}^{j}(j))  &\hh^n(\FF(-1))&\cdots &\hh^n(\FF \otimes \Omega_{\PP^n}^1(1))&\hh^n(\FF)\\
      \hline
    \hh^{n-1}(\FF \otimes \Omega_{\PP^n}^{j}(j))  &\hh^{n-1}(\FF(-1))&\cdots &\hh^{n-1}(\FF \otimes \Omega_{\PP^n}^1(1))&\hh^{n-1}(\FF)\\
      \hline
      \vdots &\vdots & & \vdots & \vdots \\
      \hline
    \hh^{1}(\FF \otimes \Omega_{\PP^n}^{j}(j))  &\hh^{1}(\FF(-1))&\cdots &\hh^{1}(\FF \otimes \Omega_{\PP^n}^1(1))&\hh^{1}(\FF)\\
    \hline
    \hh^0(\FF \otimes \Omega_{\PP^n}^{j}(j)) &\hh^0(\FF(-1))&\cdots&\hh^0(\FF \otimes \Omega_{\PP^n}^1(1))&\hh^0(\FF)\\
     \hline
     \hline
     & j=n & & j=1 & j = 0\\
     \hline
    \end{array}
  \end{equation}

The terms $L_k$ of the complex $\LL$ can be computed by taking the
direct sum of all the terms on the ``NW-SE'' diagonals, the main
diagonal corresponding to $k=0$, the first subdiagonal to $k=-1$, and
so on. 
For $j=0,\ldots,n$, each sheaf $\OO_{\PP^n}(-j)$ must be taken with
the multiplicity given by the corresponding integer in the Beilinson
table, which is exactly $\hh^s(\FF \otimes \Omega_{\PP^n}^{j}(j))$.

\subsubsection{Proof of the main theorem I. Getting the matrix}

  We need to show that conditions $(\star)$ and $(\star\star)$ are
  sufficient in order to get a matrix
  $\OO_{\PP^3}(-2)^{r+2} \to \OO_{\PP^3}(-1)^{r+2}$, whose kernel is $E(-\frac r4-2)$ and whose cokernel is 
  $E(\frac r4-1)$. In this step we are not interested in the skew-symmetry of the map: we deal with it in the next subsection.
  The proof is divided into two claims.

  The distinguished element $\beta \in
  \Ext^2(E(\frac{r}{4}-1),E(-\frac{r}{4}-2))$ corresponds to a
  $2$-term extension:
  \begin{equation}\label{2-term ext}
  \mbox{
    $0 \to  E(-\frac{r}{4}-2) \to   P_2 \xrightarrow{\p}  P_1 \to
  E(\frac{r}{4}-1) \to  0$}.
  \end{equation}

 Via the isomorphism \eqref{ext=hom}, $\beta$ can be seen as an element of
  $\Hom_{\D^b(\PP^3)}(E(\frac r4-1),E(-\frac r4-2)[2])$.
  Let $\cC:=\cC(\beta)$ be the cone of this morphism, so $\cC$ lies in the exact triangle:
  \begin{equation}\label{cono gen}
    \mbox{$E(\frac{r}{4}-1)  \to E(-\frac{r}{4}-2) [2] \to \cC \to E(\frac{r}{4}-1)[1]$}.
  \end{equation}
By taking the cohomology sequence induced by the exact triangle
\eqref{cono gen}, we see that $\cC$ has cohomology $E(-\frac r4-2)$ in
degree $-2$ and $E(\frac r4-1)$ in degree $-1$.
We have seen that $\cC$ is thus a $2$-term complex, non-zero in degree $-2$ and $-1$
only. In fact the triangle is nothing but \eqref{2-term ext}, and we have $C_{-2} = P_2$, $C_{-1}=P_1$, and $\p$ as differential.
 
We apply Beilinson Theorem \ref{beilinson} to the complex $\cC= P_2
\xrightarrow{\p} P_1$, decomposing it with respect to the
collection $\langle \OO_{\PP^3}(-3), \OO_{\PP^3}(-2), \OO_{\PP^3}(-1),
\OO_{\PP^3} \rangle$. Recall that the components of $\cC$ along the
term $\OO_{\PP^3}(-j)$ are computed by $\hh^s(\cC \otimes
\Omega^{j}(j))$. 

\vspace{0,2cm}

  \begin{claim} \label{claim1}
    The terms $\OO_{\PP^3}$ and $\OO_{\PP^3}(-3)$ do not occur in the decomposition of $\cC$.
  \end{claim}

  \begin{proof}[Proof of Claim \ref{claim1}]
    We have to show that the components of $\cC$ along $\OO_{\PP^3}$
    and $\OO_{\PP^3}(-3)$ are zero. We show that this is a direct
    consequence of $(\star)$.

   To check the statement regarding $\OO_{\PP^3}$, we need $\HH^i(\cC)=0$ for all $i$. 
   For this, take hypercohomology of (\ref{cono gen}). The vanishing
   $\HH^q(E(\frac{r}{4}-1))=\HH^{q-2}(E(-\frac{r}{4}-2))=0$ for
   $q=2,3$ tells us that  
  the only groups $\HH^i(\cC)$ that are not trivially zero fit in the exact sequence:
  $$\xymatrix@C-2.5ex@R-4ex{0 \ar[r]& \HH^{-1}(\cC) \ar[r]& \HH^0(E(\frac{r}{4}-1)) \ar[r]& \HH^2(E(-\frac{r}{4}-2)) \ar[r]& \HH^{0}(\cC)  \ar[r]&\\
    &\ar[r]&\HH^1(E(\frac{r}{4}-1)) \ar[r]& \HH^3(E(-\frac{r}{4}-2)) \ar[r]& \HH^1(\cC) \ar[r]& 0.}$$ 
  From the isomorphisms $\HH^p(E(\frac{r}{4}-1))\simeq
  \HH^{p+2}(E(-\frac{r}{4}-2))$, $p=0,1$, entailed by $(\star)$, we deduce
  the desired vanishing.

  To check that $\OO_{\PP^3}(-3)$ does not occur, we need $\HH^i(\cC(-1))=0$ for all $i$.
  By Serre duality 
  $\HH^i(E(\frac{r}{4}-2)) \simeq \HH^{3-i}(E(-\frac{r}{4}-2))^*$ and $\HH^i(E(-\frac{r}{4}-3)) \simeq \HH^{3-i}(E(\frac{r}{4}-1))^*$, 
  thus the same argument as above applies, and Claim \ref{claim1} is proved.
  \end{proof}

\vspace{0,2cm}

  Now we show that in the decomposition of $\cC$ the terms that we have not yet considered appear concentrated in one degree. 
  After Claim \ref{claim1}, the Beilinson table \eqref{quadrato2} of $\cC$ looks like this:

  \begin{equation}\label{table}
    \begin{array}{|r||c|c|c|c|}
      \hline
      \hh^3(\cC \otimes \Omega_{\PP^3}^{j}(j))&0&\lozenge&\lozenge&0\\
      \hline
      \hh^2(\cC \otimes \Omega_{\PP^3}^{j}(j))&0&\lozenge&\lozenge&0\\
      \hline
      \hh^1(\cC \otimes \Omega_{\PP^3}^{j}(j))&0&\lozenge&\lozenge&0\\
      \hline
      \hh^0(\cC \otimes \Omega_{\PP^3}^{j}(j))&0&\blacklozenge&\blacklozenge&0\\
      \hline\hline
      &j=3&j=2&j=1&j=0\\
      \hline
    \end{array}
  \end{equation}

\begin{claim} \label{claim2}
  In the Beilinson table (\ref{table}), we have $\lozenge=0$ and $\blacklozenge=r+2$
\end{claim}
  \begin{proof}[Proof of Claim \ref{claim2}]
    We show that this is a consequence of $(\star\star)$. Let us start with the term $\OO_{\PP^3}(-1)$. 
    We need $\hh^i(\cC \otimes \Omega_{\PP^3}^1(1))=0$ for all $i \neq 0$ and $\hh^0(\cC \otimes \Omega_{\PP^3}^1(1))=r+2$.
    Taking the Euler sequence tensored by $\cC$, since all terms of
    the Euler sequence are vector bundles, we obtain a
    distinguished triangle:
    \[
    \cC \otimes \Omega_{\PP^3}^1(1) \to  \cC^4 \to 
    \cC(1) \to \cC \otimes \Omega_{\PP^3}^1(1)[1],
    \]
    and we compute its hypercohomology. The vanishing $\HH^i(\cC)=0$ for all $i$ that we proved in Claim \ref{claim1} 
    implies that $\HH^i(\cC \otimes \Omega_{\PP^3}^1(1)) \simeq \HH^{i-1}(\cC(1))$, hence 
    what we want is $\HH^i(\cC(1))=0$ for all $i \neq -1$ and $\hh^{-1}(\cC(1))=r+2$.
    So let us compute hypercohomology of (\ref{cono gen}) twisted by $\OO_{\PP^3}(1)$:
    $$
    \mbox{$E( \frac{r}{4})  \to E( -\frac{r}{4}-1 ) [2] \to \cC(1) \to E( \frac{r}{4})[1].$}$$

  Analogously to what happened in the previous case, the vanishing 
  $\HH^2(E(\frac{r}{4}))=\HH^3(E(\frac{r}{4}))=\HH^0(E(-\frac{r}{4}-1))=0$ implies that
  the only groups $\HH^i(\cC(1))$ that are not trivially zero fit in the exact sequence:
  $$\xymatrix@C-2.5ex@R-4ex{0 \ar[r]& \HH^1(E(-\frac{r}{4}-1)) \ar[r]&\HH^{-1}(\cC(1)) \ar[r]& \HH^0(E(\frac{r}{4})) \ar[r]& \HH^2(E(-\frac{r}{4}-1)) \ar[r]& \HH^{0}(\cC(1)) 
    \ar[r]&\\
    &&\ar[r]&\HH^1(E(\frac{r}{4})) \ar[r]& \HH^3(E(-\frac{r}{4}-1)) \ar[r]& \HH^1(\cC(1)) \ar[r]& 0.}$$

  Now the surjection $\HH^0 ( E (\frac{r}{4} )) \epi \HH^2 ( E ( -\frac{r}{4} ))$ and the isomorphism 
  $\HH^1(E(\frac{r}{4})) \simeq \HH^3 ( E (-\frac{r}{4}-1))$ guarantee that $\HH^i(\cC(1))=0$ for $i=0$ and $1$. 
  From Riemann-Roch we get:
  $$\hh^{-1}(\cC(1))=\chi \pga E\pga \frac{r}{4}\pgc\pgc -\chi \pga E\pga -\frac{r}{4}-1\pgc\pgc=\frac{(r+4)(r+8)}{16}-\frac{r(r-4)}{16}=r+2.$$

  Finally we deal with the term $\OO_{\PP^3}(-2)$. We need to show that 
  $\HH^i(\cC \otimes \Omega_{\PP^3}^2(2))=0$ for all $i \neq 0$ and $\hh^0(\cC \otimes \Omega_{\PP^3}^2(2))=r+2$. Notice that $\Omega_{\PP^3}^2(2) \simeq \T_{\PP^3}(-2)$. 
  Then using again the (dual) Euler sequence tensored by $\cC(-1)$:
  \[
  \cC(-2) \to \cC(-1)^4 \to \cC \otimes \T_{\PP^3}(-2) \to \cC(-2)[1],
  \]
  we see that the vanishing $\HH^i(\cC(-1))=0$ for all $i$ proved in Claim \ref{claim1} implies that 
  $\HH^i(\cC \otimes \Omega_{\PP^3}^2(2)) \simeq \HH^i(\cC \otimes \T_{\PP^3}(-2)) \simeq \HH^{i+1}(\cC(-2))$. 
  We are thus left to prove $\HH^i(\cC(-2))=0$ for all $i \neq 1$ and $\hh^1(\cC(-2))=r+2$.

  We take cohomology of (\ref{cono gen}) twisted by $-2$, and notice that by Serre duality 
  $\HH^i(E(\frac{r}{4}-3)) \simeq \HH^{3-i}(E(-\frac{r}{4}-1))^*$. Hence the conditions required by $(\star\star)$ 
  and the same argument as above yield that the only non-vanishing group is $\HH^{1}(\cC(-2))$. 
  Moreover, again by Serre duality, $\hh^1(\cC(-2))=\hh^{-1}(\cC(1))=r+2$, so Claim \ref{claim2} is proved.
  \end{proof}

  \subsubsection{Proof of the main theorem II. Skew-symmetrising the matrix}

  So far, we have shown that we can decompose the cone $\cC$ explicitly as a map $\p:\OO_{\PP^3}(-2)^{r+2} \to \OO_{\PP^3}(-1)^{r+2}$, i.e. 
  the differential $\p$ is a matrix $A$ of size $r+2$ and constant rank $r$,
  that by construction will fit in a $2$-term extension of type
  (\ref{se E twist}).
  What is left to prove is that:
  \begin{claim} \label{claim3}
    The matrix $A$ is skew-symmetrizable.
  \end{claim}
  The claim means that $A$ is skew-symmetric in an appropriate basis,
  i.e., that up to composing $A$ with isomorphisms on the right and on the
  left, we get an honest skew-symmetric matrix.
  To prove this claim, we need a homological algebra lemma, that we state in greater generality
  for future reference.
  Let $F$ be a vector bundle on a smooth projective variety $X$ over $\C$, $L$
  be a line bundle on $X$.
  We have the canonical decomposition:
  \[
  \Ext^k(F,F^*\otimes L) \cong \HH^k(\wedge^2 F^* \otimes L) \oplus \HH^k(S^2 F^* \otimes L).
  \]
  We say that $\beta \in \Ext^k(F,F^*\otimes L)$ is {\it symmetric} if it belongs to $\HH^k(S^2 F^* \otimes L)$,
  and {\it skew-symmetric} if it lies in $\HH^k(\wedge^2 F^* \otimes L)$.

  \medskip

  Let $\cP$ be a bounded complex of coherent sheaves on $X$ (i.e., an object
  of $\D^b(X)$). The complex $\cP$ corresponds to an element of $\Ext^k(F,F^*\otimes L)$ if we have an exact complex:
  \begin{equation}
    \label{cP}
  \cP : 0 \to F^* \otimes L =P_{k+1} \xr{\p_k} P_k \to \cdots \to P_2 \xr{\p_1} P_1 \xr{\p_0}
  P_0 = F \to 0.  
  \end{equation}

\begin{lem}\label{lemma 0123}
  In the above setting, let $\cP$ be a complex of $k+2$ vector bundles corresponding to
  $\beta \in \Ext^k(F,F^*\otimes L)$, and assume that $\beta$ is
  symmetric.
  Moreover assume:
  \begin{align}
    \label{annullo}
    & \Ext^{>0}(P_i,P_j^* \otimes L) = 0, && \mbox{for all $i,j$;}\\
    \label{annullo2}
    & \Hom(P_i,P_{k-i}^* \otimes L)=0, &&  \mbox{for $i\le \lfloor \frac{k}{2} \rfloor$.} 
  \end{align}
  Then, up to isomorphism:
  \begin{enumerate}[i)]
  \item \label{0} if $k\equiv 0 \mod4$, the middle map of $\cP$ is symmetric;
  \item \label{1} if $k\equiv 1 \mod4$, the middle term of $\cP$ has a skew-symmetric duality;
  \item \label{2} if $k\equiv 2 \mod4$, the middle map of $\cP$ is skew-symmetric;
  \item  if $k\equiv 3 \mod4$, the middle term of $\cP$ has a symmetric duality.
  \end{enumerate}
  If $\beta$ is skew-symmetric, all signs in the above $4$
  cases must be reversed.
\end{lem}

\begin{proof}[Proof of Lemma \ref{lemma 0123}]
  We treat the symmetric case, the skew-symmetric one being analogous.

  We dualise the expression \eqref{cP} of $\cP$ and we twist by $L$ (we can do this with no harm since the $P_i$'s are
  locally free). We denote the resulting
  complex by $\cP'$. In view of the standard sign convention that we adopted for dual complexes, $\cP'$ reads:
  \[
  \cP' : 0 \to F^* \otimes L \simeq P_{0}^* \otimes L \xr{\transpose{\p}_0} P_{1}^* \otimes
  L \xr{-\transpose{\p}_1}P_{2}^* \otimes L\to \cdots 
  \to P_{k}^* \otimes L \xr{(-1)^k\transpose{\p}_k} P_{k+1}^* \otimes L \simeq F \to 0.  
  \]

  Since $\beta$ is symmetric, the class in $\Ext^k(F,F^*\otimes L)$
  corresponding to $\cP'$ is again $\beta$, so the two extensions are
  equivalent. Even though this does not imply the
  existence of an isomorphism $\cP \to \cP'$ in general, but it does under our hypothesis.
  Indeed, from \cite[Ex. 5
  Chapter III.6]{maclane} we learn that there exists a complex $\Q$ of
  $k+2$ terms, equipped with maps $\Q \to \cP$ and $\Q \to \cP'$
  lifting the identity over the terms $F^* \otimes L$ and $F$ at the
  two ends of $\cP$ and $\cP'$.
  In other words, the cones of the maps $\beta : F \to F^*\otimes L[k]$
  and $\transpose{\beta} : F \to F^*\otimes L[k]$ are quasi-isomorphic,
  hence isomorphic in $\D^b(X)$ cf. \cite[Page
  32]{libro_Huybrechts}. The situation is described in the following
  diagram: 
  $$
  \xymatrix@R-1ex{
    \cP:&0 \ar[r]& F^* \otimes L \ar[r]^{\p_{k}}& P_k \ar[r]^{\p_{k-1}}& \:\:\:\cdots\:\:\:
    \ar[r]& P_{1} \ar[r]^{\p_0} &F \ar[r]& 0\\
    \Q: &0 \ar[r]& F^* \otimes L \ar@{=}[d]\ar@{=}[u]\ar[r] & Q_k
    \ar[d]\ar[u]\ar[r]& \:\:\:\cdots \:\:\:\ar[r]& Q_1
    \ar[d]\ar[u]\ar[r] &F \ar@{=}[d]\ar@{=}[u] \ar[r]& 0\\
    \cP':&0 \ar[r]& F^* \otimes L \ar[r]_{\transpose{\p}_0}& P_{1}^* \otimes
    L \ar[r]_{-\transpose{\p}_1}&\:\:\:\cdots \:\:\:
    \ar[r]& P_{k}^* \otimes L \ar[r]_{(-1)^k\transpose{\p}_k} &F\ar[r]& 0}
  $$
  
Under the condition $\Ext^p(P_i,P_j^* \otimes L) = 0$ for $p>0$ appearing in \eqref{annullo}, the complexes $\cP$ and $\cP'$ are actually homotopic, see \cite[Lemma 1.6]{kapranov:derived}.
The further condition \eqref{annullo2} implies that the homotopy maps $P_i \ra P_{k-i}^* \otimes L$ are zero, hence $\cP$ and $\cP'$ are isomorphic complexes. 

\medskip
  So let  $\varphi : \cP \to \cP'$ be an isomorphism lifting the identity over $F$ and $F^*\otimes L$.
  For all $0 \le i \le k+1$, we have isomorphisms
  $\varphi_i : P_i \to P_{k+1-i}^* \otimes L$, with $\varphi_0=\id_F$ and
  $\varphi_{k+1}=\id_{F^* \otimes L}$, such that the following diagrams
  commute:
  \[
  \xymatrix@C+2ex{
    (D_i) && P_{k+1-i} \ar^-{\p_{k-i}}[r] \ar_-{\varphi_{k+1-i}}[d] & P_{k-i} \ar^-{\varphi_{k-i}}[d]\\
    && P_{i}^* \otimes L  \ar_-{(-1)^i\transpose{\p}_i}[r] \ar[r]  & P_{i+1}^* \otimes L
  }
  \]
  
  Let us now look at case \eqref{0}, so $k=4 h$.
  For $2 h +1 \le i \le k+1$, we can replace $\varphi_i$ with
  $\transpose{\varphi}_{k+1-i}$. We obtain squares $(D'_i)$ analogous to the $(D_i)$'s above, and the diagram has the form: 

  \[
  \xymatrix{ \ldots \ar[r] & P_{4h+1-i} \ar@{}[dr]|{(D'_i)}\ar[d]_-{\transpose{\varphi}_{i}}\ar[r]^-{\p_{4h-i}} & P_{4h-i} \ar[d]^-{\transpose{\varphi}_{i+1}}\ar[r] &\ldots \ar[r] & P_{2h+1-i} \ar@{}[dr]|{(D_i)}\ar[d]_-{\varphi_{2h+1-i}} \ar[r]^-{\p_{2h-i}} & P_{2h-i} \ar[d]^-{\varphi_{2h-i}}\ar[r]&\ldots \\
\ldots \ar[r] & P_i^* \otimes L \ar[r]_-{(-1)^i\transpose{\p}_i} & P_{i+1}^* \otimes L \ar[r] &\ldots \ar[r] & P_{2h+i}^* \otimes L 
\ar[r]_-{(-1)^i\transpose{\p}_{2h+i}} & P_{2h+1+i}^* \otimes L \ar[r] &\ldots }
  \]

  An easy computation shows that the squares in the diagram above still commute, thanks to the good behavior of sign changes. 
The diagram is symmetric with respect to the middle square $(D_{2 h}')$, that looks like this:
  \[
  \xymatrix@C+2ex{
    (D_{2h}') && P_{2h+1} \ar^-{\p_{2h}}[r] \ar_-{\transpose{\varphi}_{2h}}[d] & P_{2h} \ar^-{\varphi_{2h}}[d]\\
    && P_{2 h}^* \otimes L  \ar_-{\transpose{\p}_{2h}}[r] \ar[r]  & P_{2h+1}^* \otimes L
  }
  \]
  So up to isomorphism (i.e. up to replacing $\p_{2h}$ with
  $\varphi_{2h} \circ \p_{2 h}$), the middle differential of
  $\cP$ is a symmetric map.

  Case \eqref{2} is similar.
  Indeed, if $k = 4h +2$, we obtain new commuting diagrams $(D_i)'$ as above
  by replacing $\varphi_i$ with
  $\transpose{\varphi}_{k+1-i}$ for $2h +2 \le i \le k+1$, and the
  middle diagram is $(D_{2 h +1}')$, that yields:
  \[
  \varphi_{2h+1} \circ \p_{2h+1} = -\transpose{\p}_{2h+1} \circ  \transpose{\varphi}_{2h+1},
  \]
  so in this case the middle differential of $\cP$ is skew-symmetric
  (up to isomorphism).

  Let us now look at case \eqref{1}, so $k = 4h +1$.
  This time sign changes do not behave as well as before.
  To cope with this, we replace $\varphi_{k+1-i}$ with
  $(-1)^i\transpose{\varphi}_{i}$, for $0 \le i \le 2h$. We obtain a diagram of the form:

  \[
  \xymatrix{\ldots \ar[r] & P_{4h+2-i} \ar@{}[dr]|{(D'_i)}\ar[d]_-{(-1)^i\transpose{\varphi}_{i}}\ar[r]^-{\p_{4h+1-i}} & P_{4h+1-i} \ar[d]^-{(-1)^{i+1} \transpose{\varphi}_{i+1}}\ar[r] &\ldots   \ldots \ar[r] & P_{2h+1-i} \ar@{}[dr]|{(D_i)}\ar[d]_-{\varphi_{2h+1-i}} \ar[r]^-{\p_{2h-i}} & P_{2h-i} \ar[d]^-{\varphi_{2h-i}}\ar[r]&\ldots\\
\ldots \ar[r] & P_i^* \otimes L \ar[r]_-{(-1)^i\transpose{\p}_i} & P_{i+1}^* \otimes L \ar[r] &\ldots  \ldots \ar[r] & P_{2h+i}^* \otimes L \ar[r]_-{(-1)^i\transpose{\p}_{2h+i}} & P_{2h+1+i}^* \otimes L \ar[r] &\ldots }
  \]

  Again we get new commuting diagrams $(D_i)'$.
  The middle part of $\cP$ now gives the commuting diagram:
  \[
  \xymatrix@C+2ex{
    P_{2h+2} \ar^-{\p_{2h+1}}[r] \ar_-{\transpose{\varphi}_{2h}}[d] &    P_{2h+1} \ar^-{\p_{2h}}[r] \ar_-{\varphi_{2h+1}}[d] & P_{2h} \ar^-{\varphi_{2h}}[d]\\
    P_{2 h}^* \otimes L  \ar_-{\transpose{\p}_{2h}}[r] \ar[r] &P_{2 h+1}^* \otimes L  \ar_-{\transpose{-\p}_{2h+1}}[r] \ar[r]  & P_{2h+2}^* \otimes L
  }
  \]
  Transposing the rightmost square, and reading
  off the first square we get:
  \[
  -\transpose{\varphi}_{2 h + 1} \circ \p_{2 h + 1} =
  \transpose{\p}_{2h} \circ \transpose{\varphi}_{2h} = 
  \varphi_{2 h + 1} \circ \p_{2 h + 1}.
  \]
  This means that we can replace $\varphi_{2 h + 1}$ by $\psi = \frac 12
  (\varphi_{2 h + 1}-\transpose{\varphi}_{2 h + 1})$ without 
  spoiling the commutativity of our diagrams. Then $\psi$ will be an
  isomorphism by the five lemma, thus equipping $P_{2h+1}$ with a
  skew-symmetric duality.
  The last case is analogous to this one, so we omit it.
\end{proof}

  \begin{proof}[Proof of Claim \ref{claim3}] 
    We apply Lemma \ref{lemma 0123} to our setting. Then $X=\PP^3$, $F=E(\frac{r}{4}-1)$ and 
    $L=\OO_{\PP^3}(-3)$, so that $F^* \otimes L=E(-\frac{r}{4}-2)$. The complex $\cP$ that we are interested in is of course the cone $\cC$ corresponding 
to the distinguished element $\beta \in \Ext^2(E(\frac{r}{4}-1), E(-\frac{r}{2}-2))$, so in particular 
    $k=2$, $P_2=\OO_{\PP^3}(-2)^{r+2}$ and $P_1=\OO_{\PP^3}(-1)^{r+2}$. 
    
    Since $E$ has rank $2$, $\HH^2(\wedge^2 E(-\frac{r}{4}+1) \otimes \OO_{\PP^3}(-3))=0$, 
    meaning that all elements $\beta \in \Ext^2(E(\frac{r}{4}-1), E(-\frac{r}{2}-2))$ are symmetric. Moreover 
    conditions \eqref{annullo} and \eqref{annullo2} translate respectively in:
\[\Ext^{>0}(P_1,P_2^* \otimes L) \simeq \Ext^{>0}(P_2,P_1^* \otimes L) \simeq \HH^{>0}(\OO_{\PP^3}^{r+2})=0\] 
and in:
\[\Hom(P_1,P_1^* \otimes L) \simeq \HH^0(\OO_{\PP^3}(-1)^{r+2})=0,\]
 and are thus trivially satisfied. By part \eqref{2}, 
    the matrix $A$ that we have constructed is skew-symmetrizable.
    This concludes Claim \ref{claim3}, as well as the proof of Theorem \ref{nec&suff}.
\end{proof}

\begin{rem}
  Theorem \ref{nec&suff} is consistent with the results known for symmetric matrices. As we remarked in Section \ref{general set-up}, 
  in this case the same computation of invariants of the vector
  bundles involved holds. In \cite{Ilic_JM} the authors prove that
  if $r \ge 2$ is even,  
  then the maximal dimension of a linear space of symmetric $n \times
  n$ matrices of constant rank $r$ is $n-r+1$.  
  In other words, $4$-dimensional spaces of symmetric matrices of
  constant co-rank $2$ do not exist. 
  
  It is also worth pointing out that Claim \ref{claim3} is false on
  the projective plane $\PP^2$, simply because the group
  $\HH^2(\OO_{\PP^2}(-\frac{r}{2}-1))$ is non-zero  
  as soon as $r \ge 6$.

\end{rem}

\subsection{Simpler conditions for bundles with natural cohomology}

Imposing one further condition on the bundle $E$, namely natural cohomology, will enable us to 
simplify the requirements of $(\star)$ and $(\star\star)$. We recall that a vector bundle $E$ on $\PP^3$ 
has \emph{natural cohomology} if $\HH^i(E(t)) \neq 0$ for at most one $i$, any $t$. Remark that a rank $2$ bundle 
on $\PP^3$ with $c_1=0$, $c_2>0$ and natural cohomology is (Mumford-Takemoto) \emph{stable}, which in this setting is equivalent to the vanishing 
$\HH^0(E)=0$. Indeed, by Riemann-Roch we see that $\chi(E) \le 0$, and this, combined with the natural cohomology hypothesis, implies that the bundle has no sections.

\begin{thm}\label{conto kuz gen}
  Let $E$ be as in Theorem \ref{nec&suff}. If $E$ has natural cohomology, $(\star)$ and $(\star\star)$ reduce respectively 
  to an isomorphism:
    \begin{itemize}
    \item[$(\diamond)$]$\HH^0(E(\frac{r}{4}-1)) \simeq \HH^2(E(-\frac{r}{4}-2))$,
    \end{itemize}
    and a surjection:
    \begin{itemize}
    \item[$(\diamond\diamond)$]$\HH^0(E(\frac{r}{4})) \epi \HH^2(E(-\frac{r}{4}-1))$.
    \end{itemize}
  Hence if there exists an element $\beta \in \Ext^2(E(\frac{r}{4}-1),E(-\frac{r}{4}-2))$ 
  that induces $(\diamond)$ and $(\diamond\diamond)$, then there exists a skew-symmetric matrix of linear forms, 
  having size $r+2$ and constant rank $r$, and whose kernel is $E(-\frac{r}{4}-2)$. 
\end{thm}

\begin{proof}
  We start with $(\star)$. As remarked above, the bundle $E$ is stable. But then $\hh^0(E(-\frac{r}{4}-2))=0$,  
  and by Riemann-Roch we see that $\chi(E(-\frac{r}{4}-2)) >0$. Hence our hypothesis of natural cohomology 
  translates in the fact that $\hh^p(E(-\frac{r}{4}-2))=0$ for $p \neq 2$. It follows that once we impose that 
  $\HH^0(E(\frac{r}{4}-1)) \simeq \HH^2(E(-\frac{r}{4}-2))$, then natural cohomology will force 
  $\hh^p(E(\frac{r}{4}-1))=0$ for $p \neq 0$ and all the other requirements are trivially satisfied.
  The same reasoning works for $(\star\star)$. 
\end{proof}

\section{Instanton bundles}\label{sezione istantoni generali}

Let us take a closer look at general instantons. We call $E$ a
(mathematical) \emph{instanton bundle of charge $k$}, or simply a $k$-instanton, if $E$ is a rank $2$ 
stable vector bundle on $\PP^3$ with Chern classes $c_1=0$ and $c_2=k$, satisfying the vanishing $\HH^1(E(-2))=0$.
A general $k$-instanton $E$ has natural cohomology \cite[Thm. 0.1(i)]{HH}, 
so Theorem \ref{conto kuz gen} applies.
Here comes the main result of this Section, namely that for general
instantons the requirements needed by Theorem \ref{nec&suff} reduce to a single
non-degeneracy condition.

\begin{thm}\label{starstar gratis}
  Let $r$ be a fixed integer number of the form $12s$ or $12s-4$, $s \in \N$. Let $\Er$ be a general $k$-instanton, with $k=\frac{r(r+4)}{48}$. 
  If $\Er$ satisfies condition $(\diamond)$ of Theorem \ref{conto kuz gen}, it also satisfies condition $(\diamond\diamond)$.
\end{thm}
 
Recall that in our setting the only allowed second Chern class is
$c_2(E)=\frac{r(r+4)}{48}$. However, for consistency with instanton
literature, we still denote by $k$ the charge of $E$, keeping in mind
that $k=k(r)=\frac{r(r+4)}{48}$.

Our argument involves the (sheafified) minimal graded free
resolution of a general $k$-instanton $E$.
Let $v$ be the smallest integer such that $E(v)$ has non-zero global sections. Using 
Riemann-Roch, we compute that \cite[Rem. 8.2.3]{Hart_ist}
$v=\frac{r}{4}-1$, and we find $\hh^0(E(\frac r4 -1))=k$ and 
$\hh^1(E(\frac{r}{4}-2))=k$. Moreover, from \cite{raha} we learn what
the minimal graded free
resolution of $E$ looks like.
A direct computation, together with the assumption of natural
cohomology, then shows the following:

\begin{prop}\label{our res}
  Let $r$ be a fixed integer number of the form $12s$ or $12s-4$, $s \in \N$. Let $E$ be a general $k$-instanton, with $k=\frac{r(r+4)}{48}$. 
  Then $E$ admits the following resolution:
  \[    \mbox{$0 \longrightarrow 
      \OO_{\PP^3}(-\frac{r}{4}-1)^k\\
        \longrightarrow
      {\begin{array}{c}
          \OO_{\PP^3}(-\frac{r}{4})^b\\ 
          \oplus\\ 
          \OO_{\PP^3}(-\frac{r}{4}-1)^c
        \end{array}}\longrightarrow
      {\begin{array}{c}
          \OO_{\PP^3}(-\frac{r}{4}+1)^k\\ 
          \oplus\\ 
          \OO_{\PP^3}(-\frac{r}{4})^a
        \end{array}} \longrightarrow  E \longrightarrow  0$, \quad \mbox{where:}}
  \]
    \begin{enumerate}[i)]
    \item if $r=8$, then $a=4,\:b=0,\:c=6$;
    \item if $r=12$, then $a=4,\:b=0,\:c=10$;
    \item if $r=20$, then $a=2,\:b=0,\:c=20$;
    \item if $r \ge 24$, then $a=0$, $b=k-\frac{r}{2}-2$,
      $c=k+\frac{r}{2}$.
    \end{enumerate}
\end{prop}

We are now ready to prove the main result of this Section.

\begin{proof}[Proof of Theorem \ref{starstar gratis}]
  We use the structure of the graded module $\HH^2_*(\Er):=\bigoplus_{t \in \Z} \HH^2(\Er(t))$ to prove that
   there is a surjection: 
  \[
  \mbox{$\HH^0(\OO_{\PP^3}(1)) \otimes \HH^2(\Er(-\frac{r}{4}-2)) \epi \HH^2 (\Er(-\frac{r}{4}-1))$}.
  \]
  Combined with condition $(\diamond)$ this surjection gives us the diagram: 
  $$\xymatrix{\HH^0(\OO_{\PP^3}(1)) \otimes \HH^0(\Er(\frac{r}{4}-1)) \ar[d]
    \ar[r]^-{\simeq}& \HH^0(\OO_{\PP^3}(1)) \otimes \HH^2(\Er(-\frac{r}{4}-2))
    \ar@{->>}[d]\\ 
    \HH^0(\Er(\frac{r}{4})) \ar[r]& \HH^2(\Er(-\frac{r}{4}-1))}$$
  which implies $\HH^0(\Er(\frac{r}{4})) \epi
  \HH^2(\Er(-\frac{r}{4}-1))$, that is, condition $(\diamond\diamond)$.

  \vspace{0,2cm}

  Let us see this in detail. We call $R:=\C[x_0,x_1,x_2,x_3]$ the
  polynomial ring in $4$ variables, and for any $R$-module $\bM$ we
  denote by $\bM^\vee$ its dual as  
  $R$-module and by $\widetilde{\bM}$ its sheafification. Combining the
  results \cite[Prop. 3.2]{HR} and \cite[Prop. 1]{Decker}, we see
  that  
  if $F$ is a rank 2 vector bundle on $\PP^3$, with first Chern class
  $c_1=c_1(F)$, then the module $\bM = \HH^1_*(F)$ admits a minimal
  graded free resolution of the form:
  $$
  0\to L_4 \to L_3 \to L_2 \oplus L_0^\vee (c_1) \to L_1 \to L_0 \to \bM \to 0,
  $$
  where $F$ is the cohomology of the monad $\widetilde{L_0^\vee}(c_1) \to \widetilde{L_1} \to \widetilde{L_0}$, 
  and has a minimal graded free resolution of the form:
  \[
  0 \to \widetilde{L_4} \to \widetilde{L_3} \to \widetilde{L_2} \to F \to 0.
  \]
  Now recall that a $k$-instanton bundle is the cohomology of a monad:
  \begin{equation}\label{monade gen}
    \OO_{\PP^3}(-1)^k \to  \OO_{\PP^3}^{2k+2} \to \OO_{\PP^3}(1)^k.
  \end{equation}

  If $r \ge 24$, by Proposition \ref{our res}$(iv)$ $\Er$ admits the following resolution:
  \begin{equation}\label{res spec}
    \mbox{$0 \to  \OO_{\PP^3}(-\frac{r}{4}-2)^k \to  
      \OO_{\PP^3}(-\frac{r}{4}-1)^{k+\frac{r}{2}} \oplus \OO_{\PP^3}(-\frac{r}{4})^{k-\frac{r}{2}-2} \to  
      \OO_{\PP^3}(-\frac{r}{4}+1)^k \to  \Er \to  0$}.
  \end{equation}

  From (\ref{res spec}) and (\ref{monade gen}) we obtain the associated sequences of free $R$-modules, and by juxtaposing 
  them we resolve the first cohomology module $\bM$.
  \[
  \xymatrix@C-3ex{0 \ar[r] &R(-\frac{r}{4}-2)^k \ar[r] 
    &{\begin{array}{c}
        R(-\frac{r}{4})^{k-\frac{r}{2}-2}\\
        \oplus\\ 
        R(-\frac{r}{4}-1)^{k+\frac{r}{2}} 
      \end{array}}  \ar[r]&
    {\begin{array}{c}
        R(-\frac{r}{4}+1)^k\\
        \oplus\\
        R(-1)^k
      \end{array}}
    \ar[r] & 
    R^{2k+2}\ar[r] & R(1)^k \ar[r] & \bM \ar[r] &0.
  }
  \]

  We have that $\Ext^i_R(\bM,R)=0$ for $i \neq 4$, and $\Ext^4_R(\bM,R)$ is identified via Serre duality 
  with $\bM^*(4) \simeq \HH^2_*(\Er)$, where $\bM^*$ is the dual of $\bM$ as
  vector spaces.
  This gives the following graded resolution of the module $\HH^2_*(\Er)$:
  \[
  \mbox{$
  0 \to R(-1)^k \to  R^{2k+2} \to
    {\begin{array}{c}
        R(\frac{r}{4}-1)^k\\
        \oplus\\
        R(1)^k
      \end{array}}
    \to  
    {\begin{array}{c}
        R(\frac{r}{4})^{k-\frac{r}{2}-2}\\
        \oplus\\ 
        R(\frac{r}{4}+1)^{k+\frac{r}{2}} 
      \end{array}} 
    \to  R(\frac{r}{4}+2)^k \to  \HH^2_*(\Er) \to 0.$}
  \]
  In particular $\HH^2_*(\Er)$ is generated as graded $R$-module  by
  its elements of minimal degree $-\frac{r}{4}-2$.
  Hence all elements of $\HH^2(\Er(-\frac{r}{4}-1))$ can be obtained
  as linear combination of elements of $\HH^2(\Er(-\frac{r}{4}-2))$,
  with linear forms as coefficients. This concludes our proof for the
  case $r \ge 24$.

  The cases $r=8,12$ and $20$ are identical, once we substitute (\ref{res spec}) with the resolutions entailed by Proposition \ref{our res} $(i),(ii)$ and $(iii)$. 
  We obtain the resolution of the second cohomology module and thus surjections 
  $\HH^0(\OO_{\PP^3}(1)) \otimes \HH^2(\Er(-\tau)) \epi \HH^2(\Er(-\tau+1))$ with $\tau=4,5$ and
  $7$, for $r=8,12$ and $20$ respectively. 
\end{proof}

As a consequence of Theorem \ref{starstar gratis}, in order for a general instanton $\Er$ to produce new examples of $(r+2)\times(r+2)$ skew-symmetric 
matrices of constant co-rank 2, we only need to find an element of $\Ext^2(\Er(\frac{r}{4}-1),\Er(-\frac{r}{4}-2))$ satisfying condition $(\diamond)$ 
required by Theorem \ref{conto kuz gen}. Doing this is far from being easy. In the next Sections \ref{caso 8} and \ref{caso 12} we show how when 
$r=8$ and $r=12$ this result can be achieved.

It is worth underlining that the difficulty of finding examples increases significantly as $r$ grows. Already for 
the next two cases $r=20$ and $r=24$ the space $\Ext^2(\Er(\frac{r}{4}-1),\Er(-\frac{r}{4}-2))\simeq \HH^2(\Er \otimes \Er (-\frac{r}{2}-1))$ is expected to be zero. 
Indeed, the two would correspond to instantons of charge $10$ and $14$ respectively: in the first case we have $\chi(S^2\Er(-11))=0$, 
whereas for the latter we are in the even worse situation where $\chi(S^2 \Er (-13)) <0$.

\section{Instantons of charge two, matrices of rank eight and Westwick's example}\label{caso 8}

Here we analyse in detail the case of skew-symmetric matrices $A$ of
linear forms of size $10$, having constant rank $8$, and their
relation with $2$-instantons.

\subsection{A new point of view on Westwick's example}\label{esempio westwick}

When we began our study, the only known example of a $4$-dimensional linear space of skew-symmetric constant co-rank $2$ matrices was 
due to Westwick. It appeared with almost no explanation in \cite[page 168]{Westwick}, where the author simply exhibited the following matrix:

\begin{equation}\label{matrice westwick}
  W=  \left(
    \begin{array}{cccccccccc}
      0    &0    &0    &0    &0    &0    &0    &x_0  &x_1  &0\\
      0    &0    &0    &0    &0    &0    &x_0  &x_1  &0    &x_2\\
      0    &0    &0    &0    &0    &-x_0 &x_1  &0    &x_2  &x_3\\
      0    &0    &0    &0    &x_0  &x_1  &0    &x_2  &x_3  &0\\
      0    &0    &0    &-x_0 &0    &0    &x_2  &-x_3 &0    &0\\
      0    &0    &x_0  &-x_1 &0    &0    &x_3  &0    &0    &0\\
      0    &-x_0 &-x_1 &0    &-x_2 &-x_3 &0    &0    &0    &0\\
      -x_0 &-x_1 &0    &-x_2 &x_3  &0    &0    &0    &0    &0\\
      -x_1 &0    &-x_2 &-x_3 &0    &0    &0    &0    &0    &0\\
      0    &-x_2 &-x_3 &0    &0    &0    &0    &0    &0    &0
    \end{array}
  \right),
\end{equation}
where $x_0,x_1,x_2,x_3$ are independent variables.
The exact sequence (\ref{se E}) here reads:
\[
0 \to  E(-2) \to  \OO_{\PP^3}^{10} \xr{W} \OO_{\PP^3}^{10}(1) \to
E(3) \to 0,
\]
with $c_1(E)=0$ and $c_2(E)=2$.
The bundle $E(2)$ is globally generated, so a general global section $s$
of $E(2)$ vanishes along a smooth irreducible curve $Y$ of degree $6 =
c_2(E(2))$ having canonical sheaf $\omega_Y=\OO_Y(c_1(E(2))-4)=\OO_Y$
and genus $g=1$, i.e. $Y$ is an elliptic sextic.
We have the standard exact sequence:
\begin{equation}\label{curva ellittica}
  0 \to \OO_{\PP^3} \xr{s} E(2) \to \II_Y(4) \to 0.
\end{equation}
Moreover, since $Y$ is not contained in a quadric, computing
cohomology from (\ref{curva ellittica}) we get that $\HH^0(E)=0$,  
so $E$ is stable. 
Tensoring (\ref{curva ellittica}) by $\OO_{\PP^3}(-4)$ and computing
cohomology again we see that $\HH^1(E(-2))=0$, which means that $E$ is
an instanton bundle of charge $2$. We refer to it as \emph{Westwick
  instanton} $E_W$.

We recall that all $2$-instantons are special 't Hooft instantons
\cite[Coroll. 9.6]{Hart_ist}, where the terminology goes as follows:
a $k$-instanton $E$ is \emph{'t Hooft} if it comes via Hartshorne-Serre
correspondence from the union of $k+1$ disjoint lines, while
$E$ is called \emph{special} if $\hh^0(E(1))$ attains the maximum
possible value, namely $\hh^0(E(1))=2$.
Any 't Hooft of charge $2$ is special, because $3$ skew lines in
$\PP^3$ are always contained in a quadric.

In the case of Westwick's example, we can recover this directly. One has $\hh^0(E_W(1))=\hh^0(\II_Y(3))>0$, so a non-zero global section 
of $E$ gives:
\begin{equation}\label{3 rette}
  0 \to \OO_{\PP^3} \to E_W(1) \to \II_Z(2) \to 0,
\end{equation}
where $Z$ is the union of three skew lines in $\PP^3$, or a flat degeneration of it, 
that is, the union of a double structure on a line $\ell$ and a second line $\ell'$, or a triple structure on a line. 
The curves $Z$ and $Y$ are connected by a liaison of type $(3,3)$.

Notice that from the cohomology sequence associated to (\ref{3 rette}), it follows that 
$\HH^1(E(-2))\simeq \Ext^1(E(2),\OO_{\PP^3}) =0$, so $E(2)$ cannot be obtained as a quotient of a vector bundle of rank bigger than $2$. 
(At least not in the sense of \cite[Def. 1(iii)]{sierra_ugaglia_gg}.)

Theorem \ref{nec&suff} gives a new interpretation of the matrix (\ref{matrice westwick}): it can be seen as the explicit 
decomposition of the cone of the morphism corresponding to the extension (\ref{se gen}). 
We would like to achieve a better understanding of the instanton $E_W$. We are especially interested in finding out its orbit under the natural action of $\SL(4)$. 

Let us recall some more facts about $2$-instantons and the moduli space $M_{\PP^3}(2;0,2)$ of stable rank $2$ vector bundles $E$ on 
$\PP^3$ with $c_1=0$, $c_2=2$. Our main references are \cite{Hart_ist} and \cite{costa_ottaviani}. 
Since every section of $E(1)$ vanishes on a curve $Z$ of degree $3$ and genus $-2$, 
$E$ determines, and is uniquely determined by, the following data:
\begin{enumerate}
\item a smooth quadric $Q\subset \PP^3$;
\item one of the two rulings of $Q$, that we call the \emph{first family};
\item a linear system $g^1_3$ without base points on the first family.
\end{enumerate}

The curve $Z$ generates the quadric $Q$, the lines in the support vary
in the rulings of the first family, and the curves $Z$ describe a
base-point-free $g^1_3$ in it. 
It follows that the moduli space $M_{\PP^3}(2;0,2)$ is fibred over $\PP^9\setminus \Delta$, the space of smooth quadrics, by two copies of a variety 
$\mathcal{V} \subset \G(1,3)$. $\mathcal{V}$ is the open set formed by vector $2$-planes in $\HH^0(\OO_{\PP^1}(3))$ corresponding to the base point free $g^1_3$'s. 
In particular $M_{\PP^3}(2;0,2)$ is smooth of dimension $13$. There is a natural action of $\SL(4)$ on it induced by automorphisms of $\PP^3$. Up to this action there is a $1$-dimensional family of non-equivalent bundles: $M_{\PP^3}(2;0,2)\sslash \SL(4)\simeq \mathcal{V}\sslash\SL(2)$, and the latter is a good quotient isomorphic to $\mathbb A^1$. 
All fibres are orbits except for one, which is a union of two orbits, corresponding respectively to pencils with one triple point (dimension $3$), and two triple points (dimension $2$). 

Moreover, $E$ is completely determined by the set of its jumping lines in the Grassmannian $\G(1,3)$, which is as follows:
\begin{itemize}
\item $\ell$ is $2$-jumping, i.e. $E|_{\ell}\simeq \OO_{\ell}(-2)\oplus \OO_{\ell}(2)$, if and only if $\ell$ is a line of the second family on $Q$;
\item $\ell$ is $1$-jumping, i.e. $E|_{\ell}\simeq \OO_{\ell}(-1)\oplus \OO_{\ell}(1)$, if and only if either $\ell \subset Q$ 
  belongs to the first family and is double or triple for the curve
  $Z$ that contains it, or $\ell$ is not contained in $Q$ but meets it
  in two points of a divisor $Z$ of the $g^1_3$;
\item $\ell$ cannot be $k$-jumping for $k\ge 3$.
\end{itemize}

For the Westwick instanton $E_W$, a direct computation shows that the $2$-jumping lines can be parametrised in the form $(x_0,x_1, \alpha x_0, \alpha x_1)$, for 
$\alpha \in \C$. These lines form the second ruling in the quadric $Q$, that has therefore equation  $x_0 x_3-x_1 x_2=0$. 
The two lines parametrised by $(0,x_1,0,x_3)$ and $(x_0,0,x_2,0)$ are the only $1$-jumping lines. This proves the following:

\begin{thm}\label{orbita speciale}
  The Westwick instanton belongs to the most special orbit of $M_{\PP^3}(2;0,2)$ under the natural action of $\SL(4)$, corresponding to the $g^1_3$'s with two triple points.
\end{thm}

In Remark \ref{fibrati iso} we saw that if two matrices are equivalent under the action of $\SL(r+2) \times \SL(4)$, their associated bundles 
will not necessarily be isomorphic as vector bundles, but will belong to the same orbit under the natural action of $\SL(4)$. 
(We have seen this issue with some detail for $2$-instantons, and we refer to \cite[Lemma 4.10 and Theorem 4.13]{costa_ottaviani} for the general case.) 
Hence from Theorem \ref{orbita speciale} we can deduce that any instanton bundle $E$ of charge $2$ that does not lie in the most special orbit 
can potentially give examples of $10 \times 10$ skew-symmetric matrices of co-rank $2$ different from Westwick's. 
We will now see that this is indeed the case.

\subsection{Instanton bundles of charge two and matrices of rank eight}

We show that the construction of Theorem \ref{conto kuz gen} works for all charge
$2$ instanton bundles (not only Westwick's one, and not only general
ones).
In light of the results we saw on $M_{\PP^3}(2;0,2)\sslash \SL(4)$
and the previous remarks, this means that our method shows the
existence of a continuous family (at least of dimension
$1=\dim(M_{\PP^3}(2;0,2)\sslash \SL(4))$)
of examples of $4$-dimensional linear spaces of skew-symmetric matrices of size $10$
and constant rank $8$.

\begin{thm}\label{2-ist}
  Any $2$-instanton on $\PP^3$ induces a 
  skew-symmetric matrix of linear forms in $4$ variables having size $10$ and constant
  rank $8$.
  \end{thm}

\begin{proof}
  Let $E$ be a $2$-instanton, and let $U = \HH^0(E(1))$.
  We prove that the following natural map is surjective:
  \[
  \HH^2(E \otimes E(-5)) \simeq \Ext^2(E(1),E(-4))  \to U^* \otimes \HH^2(E(-4)).
  \]
  Then condition $(\diamond)$ follows, and Theorem \ref{conto kuz gen} applies. 
  Recall that $\HH^2(E \otimes E(-5)) \cong \HH^2(S^2 E (-5))$, and
  note that by Serre duality we are reduced to show that the following 
  natural map, corresponding to Yoneda product, is injective:
  \begin{equation}
    \label{iniett}
  U \otimes \HH^1(E) \to \HH^1(S^2 E (1)).    
  \end{equation}

  Computing cohomology, from the resolution in Proposition \ref{our res}$(i)$---or from \eqref{3 rette} if one prefers---we see that 
  $U$ has dimension $2$. Moreover the natural evaluation of global sections of $E(1)$ gives an exact sequence:
  \[
    0 \to U \otimes \OO_{\PP^3}(-1) \to E \to \OO_Q(-3,0) \to 0,
  \]
  where $Q$ is a smooth quadric in $\PP^3$, and we 
  let the zero-locus of a global sections of $E$
  be a divisor of class $|\OO_Q(3,0)|$. (See the previous Subsection \ref{esempio westwick} for more details.)
  The symmetric square $S^2$ of this exact sequence gives:
  \[
    0 \to \OO_{\PP^3}(-2) \to U \otimes E(-1) \to S^2E \to \OO_Q(-6,0) \to 0.
  \]
  Tensoring by $\OO_{\PP^3}(1)$ and taking cohomology we see that \eqref{iniett} is injective, because the cohomology of $\OO_{\PP^3}(-1)$ vanishes and $\OO_Q(-5,1)$
  has no global sections. This concludes the proof.
\end{proof}

\section{Instantons of charge four and matrices of rank twelve}\label{caso 12}

According to Westwick's computation of invariants \cite{Westwick}, the
next possible value that the rank $r$ can attain after $8$ is $12$,
so we are now looking at $14 \times 14$ skew-symmetric matrices of linear forms in $4$
variables, having constant rank $12$.

In this Section we show from a theoretical point of view that general
$4$-instantons provide examples of these matrices. More in detail, in Theorem \ref{4-ist} we prove the existence of this kind of
matrices starting from $4$-instantons with certain cohomological conditions. 
We stress that general $4$-instantons satisfy the requirements of the Theorem.

Then, in the Appendix we give an explicit matrix, with a short outline of the
strategy to produce such examples.

\begin{thm}\label{4-ist}
Let $E$ be a $4$-instanton on $\PP^3$ with natural cohomology, such that
$E(2)$ is globally generated.
Then $\Ext^2(E(2),E(-5))) \ne 0$, and a general element $\beta$ of this extension group
induces a skew-symmetric matrix of linear forms in 4 variables,
having size $14$ and constant rank $12$.
\end{thm}

\begin{proof} 
  Analogously to what we did in the case of $2$-instantons, we show that there is a surjection:
  \[
 \HH^2(S^2 E(-7)) \epi \HH^0(E(2))^* \otimes \HH^2(E(-5)).
  \]

  Set $U=\HH^0(E(2)$.
  From the assumption that $E$ has natural cohomology, and applying
  Riemann-Roch to $E(2)$, we get $\dim(U)=4$.
  Evaluation of sections 
  provides a natural map $g : U \otimes \OO_{\PP^3} \to
  E(2)$, which is surjective by assumption.
  Set $E^t(-2):=\Ker (g)$.
  We have a short exact sequence:
  \begin{equation}\label{involuzione}
    0 \to E^t(-2) \to U \otimes \OO_{\PP^3} \to E(2) \to 0.
  \end{equation}
  In \cite[Thm. 1]{D'Almeida} it is shown that $E^t$ is an instanton bundle of charge $4$, and that in fact on an open subset of 
  the moduli space of $4$-instantons, the map $E \mapsto E^t$ is an involution with no fixed points.

  Dualise sequence (\ref{involuzione}), take its second symmetric
  power $S^2$, and tensor it with $\OO_{\PP^3}(-3)$.

  We get the following $4$-term exact sequence:
  \[
  0 \to S^2E(-5) \xr{h_1} U^* \otimes E(-5) \xr{h_2} \wedge^2 U^* \otimes
  \OO_{\PP^3}(-3) \xr{h_3} \OO_{\PP^3}(1) \to 0.
  \]
  Let $\FF$ be the image of the map $h_2$ above.
  Looking at the short exact sequence:
  \[
  0 \to \FF \to \wedge^2 U^* \otimes \OO_{\PP^3}(-3) \xr{h_3}
  \OO_{\PP^3}(1) \to 0,
  \]
  and taking cohomology, we see that $\HH^2(\FF)=0$.
  Then, taking cohomology of the short exact sequence:
  \[
    0 \to S^2E(-5) \xr{h_1} U^* \otimes E(-5) \to \FF \to 0,
  \]
  we get the required natural surjective map $\HH^2(S^2 E(-7))
  \epi \HH^0(E(2))^* \otimes \HH^2(E(-5))$.
  The theorem is thus proved.
\end{proof}


\appendix
\section*{Appendix: construction of explicit examples}

We outline here an algorithmic approach to the construction of skew-symmetric matrices 
in $4$ variables, having size $14$ and constant rank $12$. The algorithm is based on the commutative algebra system {\tt Macaulay2} \cite{M2}.
Since the system runs better over finite fields, we fix a finite
field $\kk$ of characteristic different from $2$ and we work over the
polynomial ring $R=\kk[x_0\ldots,x_3]$. Conceptually, the construction holds without modification in characteristic zero.

The algorithm goes as follows.

\begin{step}
  Construct a general $4$-instanton from an elliptic curve of degree $8$ in $\PP^3$.
\end{step}

Blow up a point $p$ in $\PP^2$ and embed the blown-up plane by the
system of cubics through $p$ as a Del Pezzo surface of degree $8$ in
$\PP^8$.
A hyperplane section $C_0$ of this surface is an elliptic curve of
degree $8$ in $\PP^7$.
By a general projection into $\PP^3$, we thus get a smooth elliptic curve
$C$ of degree $8$ in $\PP^3$.
Let $I_C$ be the ideal of $C$ in $R$.
Then $I_C$ has the following minimal graded free resolution:
\[
0 \to R(-7)^4 \to R(-6)^{10} \to R(-4)^3 \oplus R(-5)^4 \to I_C \to 0.
\]

Let $\bE$ be the kernel of the induced map $R(2)^3 \to I_C(6)$. It
turns out that $\bE$ is $\HH^0_*(E^t)$ (see the previous subsection), where $E$ is the instanton
associated via the Hartshorne-Serre correspondence to the curve $C$.
We denote by $\bM$ the module $\HH^2_*(E)$; it has Hilbert function
$4,6,4$ in degrees $-5,-4,-3$.

\begin{step}
  Use the $R$-module $\bE$, together with a surjective map of
  $R$-modules $k : \bE \to \bM(-7)$ to write a $14 \times 14$ matrix
  $B$ of linear forms. 
\end{step}

To explain this step, we remark that given an element $\beta \in
\Ext^2(E(2),E(-5))$, combining the maps $\mu^2_t$ (cf. Section
\ref{idea kuz gen}) for all $t
\in \Z$ we get a map of $R$-modules:
\[
\mu^2_* : \bE = \HH^0_*(E) \to \HH^2_*(E(-7))=\bM(-7).
\]
According to our construction, $\mu^2_*$ has to be an isomorphism in
degree $2$, and an epimorphism in higher degree. This epimorphism is obtained by linearity
from the isomorphism in degree $2$, since both $\bE$ and $\bM(-7)$ are
generated in degree $2$.
We have seen in the proof of Theorem \ref{4-ist} that all isomorphisms in degree $2$
come from an element $\beta \in \Ext^2(E(2),E(-5))$.
Hence we only need a general epimorphism $k : \bE \to \bM(-7)$ and in fact a
general morphism will do.
The system {\tt Macaulay2} is capable of providing such morphism
explicitly, and expresses $k$ as a map between the generators of
$\bE$ and those of $\bM$.

To complete the argument, we resolve the truncations $\bE_{\ge 3}$ and
$\bM_{\ge 3}$.
We get presentations $u: R(-1)^{34} \to R^{20} \to \bE_{\ge 3}(3)$ and 
$v: R(-1)^{20} \to R^{6} \to \bM_{\ge 3}(3)$.
Using the map $k_3$ from the expression of $\bE_{\ge 3}$ to those of
$\bM_{\ge 3}$ induced by $k$, we get the following commutative exact
diagram:
\[
\xymatrix{
R(-1)^{34} \ar@{.>}^-{h}[d] \ar^-{u}[r] & R^{20} \ar^-{k_3}[d]\ar[r] & \bE_{\ge 3}(3) \ar^-{k}[d]\ar[r] & 0\\
R(-1)^{20} \ar[d]\ar^-{v}[r] & R^{6} \ar[d]\ar[r] & \bM_{\ge 3}(3) \ar[d]\ar[r] & 0\\
0 & 0 & 0
}
\]
The map $h$ above is induced by the diagram.
Taking sygygies of the maps above, we can complete the diagram to the following:
\[
\xymatrix{
R(-1)^{14} \ar[d] \ar^-{B}[r] & R^{14} \ar[d]\ar[r] & \bF \ar[d]\ar[r] & 0\\
R(-1)^{34} \ar^-{h}[d] \ar^-{u}[r] & R^{20} \ar^-{k_3}[d]\ar[r] & \bE_{\ge 3}(3) \ar^-{k}[d]\ar[r] & 0\\
R(-1)^{20} \ar[d]\ar^-{v}[r] & R^{6} \ar[d]\ar[r] & \bM_{\ge 3}(3) \ar[d]\ar[r] & 0\\
0 & 0 & 0
}
\]
The module $\bF$ induced above is resolved by the matrix $B$.
Since $\bE$ and $\bF$ differ only by the Artinian module $\bM$, the
induced coherent sheaves on $\PP^3$ are isomorphic, hence $B$ has constant rank $2$.

\begin{step}
  Skew-symmetrise $B$ to obtain the required matrix $A$.
\end{step}

To perform this step, we note that $B$ and the opposite transpose matrix $-\transpose{B}$
give rise to two modules $\bF$ and $\bF'$ as their cokernels, and the
associated cokernel sheaves are isomorphic.
These sheaves are stable, and thus simple, so any endomorphism of each
one of them is a multiple of the identity.
Therefore, we can consider a random morphism $\delta$ from $\bF$ to $\bF'$ to
build an isomorphism from the resolution of $\bF$ to that of $\bF'$.
Then we get an exact commutative diagram:
\[
\xymatrix{
R(-1)^{14} \ar^-{\transpose{\Delta}}[d] \ar^-{B}[r] & R^{14} \ar^{\Delta}[d]\ar[r] & \bF \ar^{\delta}[d]\ar[r] & 0\\
R(-1)^{14} \ar^-{-\transpose{B}}[r] & R^{14} \ar[r] & \bF' \ar[r] & 0,
}\]
where the matrix (of scalars) $\Delta$ is invertible.
Therefore $A = \Delta B$ is skew-symmetric, and its cokernel is
$\bF'$, so again we deduce that $A$ has constant rank $12$.

We conclude by exhibiting such a matrix $A$. To avoid cumbersome coefficients and make the matrix more readable, let us 
work on the field $\kk = \Z/7\Z$. Then $A=x_0A_0+x_1A_1+x_2A_2+x_3A_3$, where the $A_i$'s are the following skew-symmetric matrices:

\begin{scriptsize}
\[
A_0 = \bgroup\makeatletter\c@MaxMatrixCols=14\makeatother\begin{pmatrix}0&
       {-2}&
       {-1}&
       {-3}&
       {3}&
       {3}&
       {3}&
       {3}&
       1&
       {-3}&
       1&
       1&
       1&
       {-3}\\
       {2}&
       0&
       {-1}&
       {-2}&
       {3}&
       0&
       {-3}&
       {-2}&
       0&
       {3}&
       0&
       0&
       {-2}&
       0\\
       1&
       1&
       0&
       {3}&
       {-3}&
       {2}&
       {-3}&
       {-2}&
       {-3}&
       {-3}&
       {-2}&
       {-2}&
       {-3}&
       {-3}\\
       {3}&
       {2}&
       {-3}&
       0&
       {2}&
       {-3}&
       1&
       0&
       {2}&
       1&
       {3}&
       {-1}&
       0&
       1\\
       {-3}&
       {-3}&
       {3}&
       {-2}&
       0&
       {-2}&
       {-3}&
       {3}&
       {-3}&
       {-2}&
       {-1}&
       {2}&
       {2}&
       1\\
       {-3}&
       0&
       {-2}&
       {3}&
       {2}&
       0&
       {3}&
       {-3}&
       {-3}&
       1&
       {3}&
       {-1}&
       {-3}&
       1\\
       {-3}&
       {3}&
       {3}&
       {-1}&
       {3}&
       {-3}&
       0&
       1&
       {3}&
       {-1}&
       {3}&
       {-2}&
       {-1}&
       {-2}\\
       {-3}&
       {2}&
       {2}&
       0&
       {-3}&
       {3}&
       {-1}&
       0&
       {3}&
       {3}&
       {2}&
       {-3}&
       {-1}&
       0\\
       {-1}&
       0&
       {3}&
       {-2}&
       {3}&
       {3}&
       {-3}&
       {-3}&
       0&
       {-1}&
       {3}&
       {-2}&
       {-3}&
       {-3}\\
       {3}&
       {-3}&
       {3}&
       {-1}&
       {2}&
       {-1}&
       1&
       {-3}&
       1&
       0&
       {-1}&
       {-2}&
       {2}&
       {-1}\\
       {-1}&
       0&
       {2}&
       {-3}&
       1&
       {-3}&
       {-3}&
       {-2}&
       {-3}&
       1&
       0&
       1&
       {-3}&
       {-1}\\
       {-1}&
       0&
       {2}&
       1&
       {-2}&
       1&
       {2}&
       {3}&
       {2}&
       {2}&
       {-1}&
       0&
       {3}&
       {-2}\\
       {-1}&
       {2}&
       {3}&
       0&
       {-2}&
       {3}&
       1&
       1&
       {3}&
       {-2}&
       {3}&
       {-3}&
       0&
       1\\
       {3}&
       0&
       {3}&
       {-1}&
       {-1}&
       {-1}&
       {2}&
       0&
       {3}&
       1&
       1&
       {2}&
       {-1}&
       0\\
       \end{pmatrix}\egroup\]
\[A_1=\bgroup\makeatletter\c@MaxMatrixCols=14\makeatother\begin{pmatrix}0&
       {-2}&
       {-2}&
       {3}&
       {-3}&
       0&
       {-2}&
       {-3}&
       {3}&
       {-1}&
       {2}&
       0&
       {2}&
       {-3}\\
       {2}&
       0&
       {3}&
       {-1}&
       1&
       {2}&
       {-3}&
       {-1}&
       {-2}&
       {-1}&
       {-1}&
       {-3}&
       1&
       {2}\\
       {2}&
       {-3}&
       0&
       {-2}&
       1&
       1&
       1&
       {-1}&
       {2}&
       {-3}&
       0&
       {-3}&
       {2}&
       {-3}\\
       {-3}&
       1&
       {2}&
       0&
       {2}&
       {-1}&
       1&
       {-2}&
       {-1}&
       {-2}&
       1&
       {2}&
       {2}&
       {-3}\\
       {3}&
       {-1}&
       {-1}&
       {-2}&
       0&
       1&
       {-1}&
       1&
       {-2}&
       0&
       {-1}&
       {2}&
       0&
       0\\
       0&
       {-2}&
       {-1}&
       1&
       {-1}&
       0&
       {3}&
       0&
       {-2}&
       {2}&
       {2}&
       {-3}&
       {-3}&
       1\\
       {2}&
       {3}&
       {-1}&
       {-1}&
       1&
       {-3}&
       0&
       {-3}&
       {2}&
       {3}&
       {-1}&
       {-2}&
       {-2}&
       {3}\\
       {3}&
       1&
       1&
       {2}&
       {-1}&
       0&
       {3}&
       0&
       1&
       1&
       {3}&
       0&
       {3}&
       {-1}\\
       {-3}&
       {2}&
       {-2}&
       1&
       {2}&
       {2}&
       {-2}&
       {-1}&
       0&
       {2}&
       {-1}&
       {-3}&
       1&
       {2}\\
       1&
       1&
       {3}&
       {2}&
       0&
       {-2}&
       {-3}&
       {-1}&
       {-2}&
       0&
       {-1}&
       1&
       {3}&
       {-1}\\
       {-2}&
       1&
       0&
       {-1}&
       1&
       {-2}&
       1&
       {-3}&
       1&
       1&
       0&
       {-3}&
       {2}&
       {-3}\\
       0&
       {3}&
       {3}&
       {-2}&
       {-2}&
       {3}&
       {2}&
       0&
       {3}&
       {-1}&
       {3}&
       0&
       {3}&
       {-1}\\
       {-2}&
       {-1}&
       {-2}&
       {-2}&
       0&
       {3}&
       {2}&
       {-3}&
       {-1}&
       {-3}&
       {-2}&
       {-3}&
       0&
       0\\
       {3}&
       {-2}&
       {3}&
       {3}&
       0&
       {-1}&
       {-3}&
       1&
       {-2}&
       1&
       {3}&
       1&
       0&
       0\\
       \end{pmatrix}\egroup\]
\[A_2=\bgroup\makeatletter\c@MaxMatrixCols=14\makeatother\begin{pmatrix}0&
       {2}&
       {2}&
       {-3}&
       {3}&
       {2}&
       {-1}&
       {-1}&
       1&
       1&
       0&
       {2}&
       {-3}&
       {-2}\\
       {-2}&
       0&
       {-2}&
       {3}&
       {3}&
       {-1}&
       1&
       {-1}&
       {-3}&
       {-2}&
       1&
       {-3}&
       {-2}&
       {-2}\\
       {-2}&
       {2}&
       0&
       1&
       {3}&
       1&
       {3}&
       {2}&
       {2}&
       {3}&
       {2}&
       1&
       0&
       {-3}\\
       {3}&
       {-3}&
       {-1}&
       0&
       {-3}&
       {-1}&
       1&
       {-3}&
       {3}&
       {-1}&
       {-3}&
       {2}&
       {-3}&
       1\\
       {-3}&
       {-3}&
       {-3}&
       {3}&
       0&
       {3}&
       {-2}&
       {-3}&
       {3}&
       1&
       {-3}&
       {-1}&
       0&
       {2}\\
       {-2}&
       1&
       {-1}&
       1&
       {-3}&
       0&
       {3}&
       {-2}&
       0&
       {-2}&
       0&
       {-2}&
       {-2}&
       {-2}\\
       1&
       {-1}&
       {-3}&
       {-1}&
       {2}&
       {-3}&
       0&
       0&
       {2}&
       {-1}&
       {-2}&
       {-3}&
       {2}&
       {-2}\\
       1&
       1&
       {-2}&
       {3}&
       {3}&
       {2}&
       0&
       0&
       {-2}&
       0&
       {2}&
       {-2}&
       0&
       {3}\\
       {-1}&
       {3}&
       {-2}&
       {-3}&
       {-3}&
       0&
       {-2}&
       {2}&
       0&
       {-1}&
       {-1}&
       {-1}&
       0&
       {-1}\\
       {-1}&
       {2}&
       {-3}&
       1&
       {-1}&
       {2}&
       1&
       0&
       1&
       0&
       {-2}&
       {3}&
       {-2}&
       {3}\\
       0&
       {-1}&
       {-2}&
       {3}&
       {3}&
       0&
       {2}&
       {-2}&
       1&
       {2}&
       0&
       {-3}&
       {3}&
       {-1}\\
       {-2}&
       {3}&
       {-1}&
       {-2}&
       1&
       {2}&
       {3}&
       {2}&
       1&
       {-3}&
       {3}&
       0&
       0&
       {-1}\\
       {3}&
       {2}&
       0&
       {3}&
       0&
       {2}&
       {-2}&
       0&
       0&
       {2}&
       {-3}&
       0&
       0&
       {3}\\
       {2}&
       {2}&
       {3}&
       {-1}&
       {-2}&
       {2}&
       {2}&
       {-3}&
       1&
       {-3}&
       1&
       1&
       {-3}&
       0\\
       \end{pmatrix}\egroup\]
\[A_3=\bgroup\makeatletter\c@MaxMatrixCols=14\makeatother\begin{pmatrix}0&
       {-3}&
       {2}&
       {-3}&
       {-1}&
       {-1}&
       {3}&
       {-2}&
       {3}&
       {3}&
       {3}&
       0&
       {-3}&
       {-3}\\
       {3}&
       0&
       {-3}&
       1&
       1&
       {2}&
       {-1}&
       {-3}&
       {-1}&
       0&
       {3}&
       {-3}&
       0&
       {-1}\\
       {-2}&
       {3}&
       0&
       1&
       {-1}&
       0&
       {-1}&
       {-2}&
       {3}&
       0&
       {-1}&
       {-2}&
       1&
       {-2}\\
       {3}&
       {-1}&
       {-1}&
       0&
       {3}&
       {2}&
       {-1}&
       0&
       1&
       {-3}&
       {-3}&
       {-1}&
       {-1}&
       {3}\\
       1&
       {-1}&
       1&
       {-3}&
       0&
       {3}&
       {3}&
       0&
       0&
       {-3}&
       {-3}&
       {3}&
       {2}&
       {3}\\
       1&
       {-2}&
       0&
       {-2}&
       {-3}&
       0&
       {-3}&
       {3}&
       {3}&
       {-3}&
       {-3}&
       {-2}&
       1&
       1\\
       {-3}&
       1&
       1&
       1&
       {-3}&
       {3}&
       0&
       0&
       {3}&
       {2}&
       0&
       {-3}&
       {2}&
       0\\
       {2}&
       {3}&
       {2}&
       0&
       0&
       {-3}&
       0&
       0&
       0&
       1&
       {-1}&
       {-2}&
       1&
       1\\
       {-3}&
       1&
       {-3}&
       {-1}&
       0&
       {-3}&
       {-3}&
       0&
       0&
       {3}&
       {2}&
       {-3}&
       {-1}&
       {-1}\\
       {-3}&
       0&
       0&
       {3}&
       {3}&
       {3}&
       {-2}&
       {-1}&
       {-3}&
       0&
       {-2}&
       {-2}&
       {-3}&
       {-2}\\
       {-3}&
       {-3}&
       1&
       {3}&
       {3}&
       {3}&
       0&
       1&
       {-2}&
       {2}&
       0&
       1&
       {2}&
       {-3}\\
       0&
       {3}&
       {2}&
       1&
       {-3}&
       {2}&
       {3}&
       {2}&
       {3}&
       {2}&
       {-1}&
       0&
       {2}&
       {-2}\\
       {3}&
       0&
       {-1}&
       1&
       {-2}&
       {-1}&
       {-2}&
       {-1}&
       1&
       {3}&
       {-2}&
       {-2}&
       0&
       {-3}\\
       {3}&
       1&
       {2}&
       {-3}&
       {-3}&
       {-1}&
       0&
       {-1}&
       1&
       {2}&
       {3}&
       {2}&
       {3}&
       0\\
       \end{pmatrix}\egroup\]
\end{scriptsize}

\bibliographystyle{amsalpha}
\bibliography{biblioada}

\def\cprime{$'$}
\providecommand{\bysame}{\leavevmode\hbox to3em{\hrulefill}\thinspace}
\providecommand{\MR}{\relax\ifhmode\unskip\space\fi MR }
\providecommand{\MRhref}[2]{%
  \href{http://www.ams.org/mathscinet-getitem?mr=#1}{#2}
}
\providecommand{\href}[2]{#2}
\begin{thebibliography}{AHDM78}

\bibitem[AHDM78]{adhm}
M.~F. Atiyah, N.~J. Hitchin, V.~G. Drinfel{\cprime}d, and Y.~I. Manin,
  \emph{Construction of instantons}, Phys. Lett. A \textbf{65} (1978), no.~3,
  185--187.

\bibitem[Be{\u\i}78]{beilinson}
A.~A. Be{\u\i}linson, \emph{Coherent sheaves on {${\bf P}^{n}$} and problems in
  linear algebra}, Funktsional. Anal. i Prilozhen. \textbf{12} (1978), no.~3,
  68--69.

\bibitem[C{\u{a}}l05]{caldararu:skimming}
A.~C{\u{a}}ld{\u{a}}raru, \emph{Derived categories of sheaves: a skimming},
  Snowbird lectures in algebraic geometry, Contemp. Math., vol. 388, Amer.
  Math. Soc., Providence, RI, 2005, pp.~43--75. \MR{2182889 (2006h:14022)}

\bibitem[CO02]{costa_ottaviani}
L.~Costa and G.~Ottaviani, \emph{Group action on instanton bundles over {$\Bbb
  P^3$}}, Math. Nachr. \textbf{246/247} (2002), 31--46.

\bibitem[D'A00]{D'Almeida}
J.~D'Almeida, \emph{Une involution sur un espace de modules de fibr\'es
  instantons}, Bull. Soc. Math. France \textbf{128} (2000), no.~4, 577--584.

\bibitem[Dec90]{Decker}
W.~Decker, \emph{Monads and cohomology modules of rank 2 vector bundles},
  Compositio Math. \textbf{76} (1990), no.~1--2, 7--17.

\bibitem[EH88]{Eisenbud_Harris}
D.~Eisenbud and J.~Harris, \emph{Vector spaces of matrices of low rank}, Adv.
  in Math. \textbf{70} (1988), no.~2, 135--155.

\bibitem[FM11]{Fania_Mezzetti}
M.L. Fania and E.~Mezzetti, \emph{Vector spaces of skew-symmetric matrices of
  constant rank}, Linear Algebra Appl. \textbf{434} (2011), 2383--2403.

\bibitem[GM96]{Gelfand_Manin}
S.I. Gelfand and Y.I. Manin, \emph{Methods of homological algebra},
  Springer-Verlag, Berlin, 1996, Translated from the 1988 Russian original.

\bibitem[GS]{M2}
D.R. Grayson and M.E. Stillman, \emph{Macaulay2, a software system for research
  in algebraic geometry}, Available at http://www.math.uiuc.edu/Macaulay2/.

\bibitem[Har77]{Har}
R.~Hartshorne, \emph{Algebraic geometry}, Graduate Texts in Mathematics,
  no.~52, Springer-Verlag, New York-Heidelberg, 1977.

\bibitem[Har78]{Hart_ist}
\bysame, \emph{Stable vector bundles of rank {$2$} on {${\bf P}^{3}$}}, Math.
  Ann. \textbf{238} (1978), no.~3, 229--280.

\bibitem[HH82]{HH}
R.~Hartshorne and A.~Hirschowitz, \emph{Cohomology of a general instanton
  bundle}, Ann. scient. \'Ec. Norm. Sup. \textbf{15} (1982), 365--390.

\bibitem[HR91]{HR}
R.~Hartshorne and A.P. Rao, \emph{Spectra and monads of stable bundles}, J.
  Math. Kyoto Univ. \textbf{31} (1991), no.~3, 789--806.

\bibitem[Huy06]{libro_Huybrechts}
D.~Huybrechts, \emph{Fourier-{M}ukai transforms in algebraic geometry}, Oxford
  Mathematical Monographs, The Clarendon Press Oxford University Press, Oxford,
  2006.

\bibitem[IL99]{Ilic_JM}
B.~Ilic and J.M. Landsberg, \emph{On symmetric degeneracy loci, spaces of
  symmetric matrices of constant rank and dual varieties}, Math. Ann.
  \textbf{314} (1999), no.~1, 159--174.

\bibitem[Kap88]{kapranov:derived}
M.~M. Kapranov, \emph{On the derived categories of coherent sheaves on some
  homogeneous spaces}, Invent. Math. \textbf{92} (1988), no.~3, 479--508.

\bibitem[ML95]{maclane}
S.~Mac~Lane, \emph{Homology}, Classics in Mathematics, Springer-Verlag, Berlin,
  1995, Reprint of the 1975 edition.

\bibitem[MM05]{Manivel_Mezzetti}
L.~Manivel and E.~Mezzetti, \emph{On linear spaces of skew–-symmetric
  matrices of constant rank}, Manuscripta Math. \textbf{117} (2005), no.~3,
  319--331.

\bibitem[New81]{Newstead}
P.~E. Newstead, \emph{Invariants of pencils of binary cubics}, Math. Proc.
  Cambridge Philos. Soc. \textbf{89} (1981), no.~2, 201--209.

\bibitem[Rah97]{raha}
O.~Rahavandrainy, \emph{R\'esolution des fibr\'es instantons g\'en\'eraux}, C.
  R. Acad. Sci. Paris S\'er. I Math. \textbf{325} (1997), no.~2, 189--192.

\bibitem[SU09]{sierra_ugaglia_gg}
J.C. Sierra and L.~Ugaglia, \emph{On globally generated vector bundles on
  projective spaces}, J. Pure Appl. Algebra \textbf{213} (2009), no.~11,
  2141--2146.

\bibitem[Syl86]{sylvester}
J.~Sylvester, \emph{On the dimension of spaces of linear transformations
  satisfying rank conditions}, Linear Algebra Appl. \textbf{78} (1986), 1--10.

\bibitem[Wes87]{Westwick1}
R.~Westwick, \emph{Spaces of matrices of fixed rank}, Linear and Multilinear
  Algebra \textbf{20} (1987), no.~2, 171--174.

\bibitem[Wes96]{Westwick}
\bysame, \emph{Spaces of matrices of fixed rank. {I}{I}}, Linear Algebra Appl.
  \textbf{235} (1996), 163--169.

\end{thebibliography}

\end{document}